\documentclass[a4paper,11pt]{article}

\usepackage{hyperref}
\hypersetup
{
    colorlinks,
    citecolor=green,
    filecolor=black,
    linkcolor=blue,
    urlcolor=blue
}
\hypersetup{linktocpage}

\usepackage{mathrsfs}

\usepackage{authblk}
\usepackage{blindtext}

\usepackage{graphicx}

\usepackage{amsbsy}
\usepackage{latexsym}
\usepackage{amsfonts}
\usepackage{amssymb}
\usepackage[usenames]{color}
\usepackage{amsmath,amsthm}
\usepackage{enumerate}
\usepackage{authblk}
\usepackage{tikz}
\usetikzlibrary{arrows}
\usepackage{mathdots}
\usepackage{mathtools}
\usepackage{tikz}
\usepackage{cases}

\usepackage{amscd}
\usepackage{amstext}

\usepackage{graphics}
\usepackage{graphicx}
\usepackage{dsfont}

\usepackage{color}

\newtheorem{theorem}{Theorem}[section]
\newtheorem{lemma}[theorem]{Lemma}
\newtheorem{prop}[theorem]{Proposition}
\newtheorem{cor}[theorem]{Corollary}

\theoremstyle{definition}

\theoremstyle{remark}
\newtheorem{remark}[theorem]{Remark}

\newcommand{\cL}{{\mathcal{L}}}

\newcommand{\dd}{\mathrm{d}}

\newcommand{\Chi}{\raise .3ex
\hbox{\large $\chi$}} 

\newcommand{\T}{{\mathbb{T}}}
\newcommand{\R}{\mathbb{R}}

\newcommand{\N}{\mathbb{N}}
\newcommand{\Z}{\mathbb{Z}}
\newcommand{\C}{\mathbb{C}}

\newcommand{\zz}{\mathbb{Z}}
\newcommand{\re}{\mathbb{R}}
\newcommand{\tor}{\mathbb{T}}

\newcommand\mix{\mathop{\rm mix}}

\usepackage[section]{algorithm}
\usepackage{algorithmicx}
\usepackage{algpseudocode}
\algrenewcommand\algorithmicrequire{\makebox[46pt][l]{\textrm{required:}}}
\algrenewcommand\algorithmicensure{\makebox[46pt][l]{\textrm{output:}}}
\algrenewcommand\algorithmicfunction{\textrm{function}}
\algrenewcommand\algorithmicwhile{\textrm{while}}
\algrenewcommand\algorithmicdo{}
\algrenewcommand\algorithmicend{\textrm{end}}
\algrenewcommand\algorithmicforall{\textrm{for all}}
\algrenewcommand\algorithmicfor{\textrm{for}}
\algrenewcommand\algorithmicrepeat{\textrm{repeat}}
\algrenewcommand\algorithmicuntil{\textrm{until}}
\algrenewcommand\algorithmicif{\textrm{if}}
\algrenewcommand\algorithmicthen{\textrm{then}}
\algrenewcommand\algorithmicelse{\textrm{else}}

\newcommand{\bk}{{\boldsymbol{k}}}

\newcommand{\bl}{{\boldsymbol{l}}}

\newcommand{\bx}{{\boldsymbol{x}}}

\newcommand{\by}{{\boldsymbol{y}}}

\newcommand{\balpha}{{\boldsymbol{\alpha}}}

\usepackage{todonotes}

\newcommand{\be}{\begin{equation}}
\newcommand{\ee}{\end{equation}}
\newcommand{\beq}{\begin{eqnarray}}
\newcommand{\beqq}{\begin{eqnarray*}}
\newcommand{\eeq}{\end{eqnarray}}
\newcommand{\eeqq}{\end{eqnarray*}}

\numberwithin{equation}{section}

\title{{$s$-Numbers of Embeddings of Weighted Wiener Algebras}}

\author[a]{Van Dung Nguyen}

\author[a]{Van Kien Nguyen	\footnote{Corresponding author}}

\author[b]{Winfried Sickel}
\affil[a]{Department of Mathematical Analysis, University of Transport and Communications
	\protect\\	No.3 Cau Giay Street, Lang Thuong Ward, Dong Da District,
	Hanoi, Vietnam
	\protect\\
	Email: dungnv@utc.edu.vn,\ kiennv@utc.edu.vn
}

\affil[b]{Friedrich-Schiller-Universit\"at Jena, 
Ernst-Abbe-Platz 2, 07737 Jena, Germany
	\protect\\ 
	Email: winfried.sickel@uni-jena.de
}

\date{\today}
\begin{document}
 
\maketitle
\begin{abstract}  
In this paper we study
the asymptotic behavior of Kolmogorov, approximation, Bernstein and Weyl numbers of embeddings $ \mathcal{A}^{s,r}_{\mix}(\T^d) \to L_2(\T^d)$ 
 and $\mathcal{A}^{s,r}_{\mix}(\T^d) \to \mathcal{A}(\T^d)$, where $\mathcal{A}^{s,r}_{\mix}(\T^d)$ is a weighted Wiener algebra of mixed smoothness $s$ and  
 $\mathcal{A}(\T^d)$ is the Wiener algebra itself, both defined on  the $d$-dimensional torus $\T^d$.
Our main interest consists in the calculation of the associated asymptotic constants. 


\end{abstract}


\section{Introduction}


Let $\omega=(\omega(\bk))_{\bk\in \Z^d}$  be a sequence of positive numbers.
In this paper we are concerned with the  asymptotic behavior of certain  $s$-numbers with respect to the embedding of a weighted Wiener algebra
$\mathcal{A}_{\omega}(\T^d)$ on the $d$-dimensional torus $\T^d$ into either the 
Wiener algebra $\mathcal{A}(\T^d)$ itself or  $L_p (\T^d)$, $2\le p \le \infty$.
Most interesting for us will be the case of dominating mixed smoothness when the weight is given by 
 $ \omega_{s,r}=(\omega_{s,r}(\bk))_{\bk \in \Z^d}$, where $s>0$ and $\omega_{s,r}(\bk)= \prod_{i=1}^d\big(1+|k_i|^r\big)^{s/r}$ if $0<r<\infty$,  $\omega_{s,\infty}(\bk) =  \prod_{i=1}^d \max (1,|k_i|)^{s}$. In this context the spaces $\mathcal{A}^{s,r}_{\mix}(\T^d):=\mathcal{A}_{\omega_{s,r}} (\T^d)$ are counterparts of the well-known 
Sobolev spaces $H^{s,r}_{\mix} (\T^d)$ of dominating mixed smoothness. In what follows we shall call $\mathcal{A}^{s,r}_{\mix}(\T^d)$ the weighted Wiener algebra of mixed smoothness $s$.  The parameter 
$r\in (0,\infty]$ does not influence the set of functions, it refers to a family of  equivalent norms.  
Our motivation of studying the asymptotic behavior of $s$-numbers of embeddings of $\mathcal{A}^{s,r}_{\mix}(\T^d)$ comes from recent papers on high-dimensional approximation \cite{BKUV17,CKS16,KPV15,KV19}. The authors in \cite{BKUV17,CKS16} studied approximation numbers or sampling widths of 
embeddings of Sobolev spaces with mixed smoothness in the norm of $\mathcal{A}(\T^d)$,  whereas \cite{KPV15,KV19} dealt  with 
sampling widths at rank-1 lattice nodes of embeddings from $\mathcal{A}^{s,r}_{\mix}(\T^d)$ into $L_\infty(\T^d)$ or into isotropic Sobolev spaces. In all the mentioned papers, only \cite{CKS16} gave the correct asymptotic order of approximation numbers of the embedding $H^{s,r}_{\mix} (\T^d)\hookrightarrow \mathcal{A}(\T^d)$, while the other either gave an upper or a lower bound for sampling widths.

In the present paper, in  specific situations we shall not only investigate the optimal order of the decay of the $s$-numbers (as classically done) but we shall determine 
the asymptotic constant as well. This sheds some light not only on the dependence on $n$, but also on the dependence on $s,r$ 
and in particular on $d$. Our main results are the following. 
For  $0<s<\infty$, $0<r\leq \infty$  and $d\in \N$ we have 

\begin{equation}\label{eq-intro-01}
\begin{aligned}
 \lim_{n \to \infty} \, \frac{s_{n} \big(id: \, \mathcal{A}^{s,r}_{\mix} (\T^d) \to \mathcal{A} (\T^d)\big)}{n^{-s}(\ln n)^{s(d-1)} \, } 
&= \bigg(\frac{2^{d}}{(d-1)!}\bigg)^s \, , 
\\
&
\\
\lim\limits_{n\to \infty} \frac{a_n\big(id: \mathcal{A}_{\mix}^{s,r}(\T^d) \to L_2(\T^d)\big)}{n^{-s}(\ln n)^{s(d-1)}}
& = 
\lim\limits_{n\to \infty} \frac{d_n\big(id: \mathcal{A}_{\mix}^{s,r}(\T^d) \to L_2(\T^d)\big)}{n^{-s}(\ln n)^{s(d-1)}}
\\
& =
\bigg(\frac{2s}{2s+1}\bigg)^s\bigg( \frac{2^d}{(d-1)!}\bigg)^s\, ,
\end{aligned}
\end{equation}
and
\begin{equation}\label{eq-intro-01.5}
\begin{aligned}
\lim\limits_{n\to \infty} \frac{b_n\big(id: \mathcal{A}_{\mix}^{s,r}(\T^d) \to L_2(\T^d)\big)}{n^{-s-\frac{1}{2}}(\ln n)^{s(d-1)}} 
&= 
\lim\limits_{n\to \infty} \frac{ x_n\big(id: \mathcal{A}_{\mix}^{s,r}(\T^d) \to L_2(\T^d)\big)}{n^{-s-\frac{1}{2}}(\ln n)^{s(d-1)}}
\\
& = 
\sqrt{ 2s+1}\bigg( \frac{2^d}{(d-1)!}\bigg)^s \,.
\end{aligned}
\end{equation}
Here $s_n \in \{a_n, b_n, d_n, x_n\}$.
By $a_n$ we denote the approximation numbers (linear widths), by $b_n$
the Bernstein numbers, $d_n$ refers to the Kolmogorov numbers and  $x_n$ to the Weyl numbers.
From these relations it  immediately follows that in the first two  cases 
the optimal order of the decay (asymptotic behavior)
is given by  $n^{-s}\, (\ln n)^{s(d-1)}$, $n \in \N$, whereas in the third 
case we have $n^{-s-1/2}\, (\ln n)^{s(d-1)}$, $n \in \N$.
In other words, for approximation numbers there is not much difference in the approximation with respect to  the norm 
of the Wiener algebra $\mathcal{A}(\T^d)$ or with respect to the norm in $L_2 (\T^d)$, but for Bernstein and Weyl numbers it is. 
The value of these limits above we shall call asymptotic constant.
In case when the target space of the embedding is $L_p (\T^d)$, $2<p\le\infty$, we are not able to determine the asymptotic constant, we do not even know whether 
the limit exists. 
However, we will prove two-sided estimates for the approximation numbers (and sometimes also other $s$-numbers) with constants independent of $n$
and $d$. For $1\le p < 2$ we do not know  so much, we refer to \cite{CKS19} for some results, mainly based on duality.

The original motivation to consider the $d$-dependence of the behavior of approximation numbers 
comes from certain needs of numerical analysis on high-dimensional 
approximation which has been the object of an intensive study recently. We refer to  
Bungartz, Griebel \cite[Theorem\ 3.8]{BuGr04} as well as to 
Schwab, S\"uli, and  Todor \cite{SST08}. In both papers the  non-periodic situation is considered.
In the periodic situation, let us refer to 
Dinh D{\~u}ng, Ullrich \cite{DU13}, Chernov, Dinh {D}\~ung \cite{ChD16}, Krieg \cite{Kr18}, K\"uhn, Mayer, Ullrich \cite{KMU16}
K\"uhn, Sickel, Ullrich \cite{KSU14,KSU15,KSU20} and Cobos, K\"uhn, Sickel \cite{CKS16,CKS19}.
In all these quoted papers  the asymptotic behavior as well as the preasymptotic behavior of 
the approximation numbers  of the embedding of a weighted Hilbert space $F_\omega (\T^d)$ either into $L_2 (\T^d)$ or into $L_\infty (\T^d)$
has been investigated.

Let us give a short overview of what is known  about the $d$-dependence  of  certain $s$-numbers at this moment.
For simplicity we focus our attention on approximation numbers. This situation is illustrated in Figure \ref{fig1}. Each arrow in this  figure is  interpreted as 
the existence of a  sharp two-sided estimate for approximation numbers of a corresponding embedding with constants independent of $n$ and $d$.
The parameter $p_0$ in Figure \ref{fig1} refers to the target space $L_{p_0} (\T^d)$ except $p_0=\infty$. In the case $s_0=0$ and $p_0=\infty$ the point $(s_0,1/p_0)= (0,0)$
has to be identified with either $L_\infty (\T^d)$ or  
$\mathcal{A} (\T^d)$.


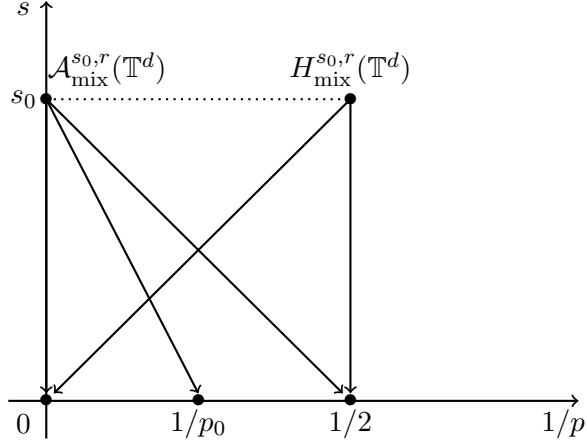
\begin{figure} 
	\begin{center}
\begin{tikzpicture}[thick]
\draw[->] (-1.5,0) -- (6,0) ;
\draw[->] (-1,-0.5) -- (-1,5.3) ;

\draw[->] (-1,4) -- (-1,0.1) ;

\draw[->] (-1,4) -- (2.9,0.1) ;
\draw[->] (3,4) -- (3,0.1) ;

\draw[->] (3,4) -- (-0.9,0.1) ;

\draw[->] (-1,4) -- (1,0.1) ;

\draw[dotted] (-1,4) -- (3,4) ;

\node at (3, 4.4) {$H^{s_0,r}_{\mix} (\T^d)$} ;
\node at (-0.2, 4.4) {$\mathcal{A}^{s_0,r}_{\mix} (\T^d)$} ;

\node at (-1.3,5.2) {$s$} ;

\node at (-1,4) {$\bullet$} ;
\node at (3,0) {$\bullet$} ;
\node at (-1,0) {$\bullet$} ;
\node at (1,0) {$\bullet$} ;
\node at (3,4) {$\bullet$} ;
\node at (-1.3,4) {$s_0$} ;
\node at (3,-0.3) {$1/2$} ;
\node at (-1.3,-0.3) {$0$} ;
\node at (5.8,-0.3) {$1/p$} ;
\node at (1,-0.3) {$1/p_0$} ;
\end{tikzpicture}
\caption{\small Illustration of embeddings with known  sharp two-sided estimates for the approximation numbers}
\label{fig1}
\end{center}
\end{figure}
It would be desirable to  replace 
$\mathcal{A}^{s_0,r}_{\mix} (\T^d)$ by  the classical (mixed-type) H\"older-Zygmund space ${C}^{s_0}_{\mix} (\T^d) $, 
the Nikol'skij-Besov  $S^{s_0}_{\infty,\infty} B(\T^d)$ or the Sobolev space $S^{s_0}_{\infty}W(\T^d)$ of dominating mixed smoothness.
However, with this respect we do not have a complete  answer. In the present paper, for Bernstein and Weyl numbers and $m\in \N$ we prove that
\begin{align*}
\lim\limits_{n\to \infty} \frac{ b_n\big(id: \mathring{S}^m_{\infty}W(\T^d) \to L_2(\T^d)\big)}{n^{-m-\frac{1}{2}}(\ln n)^{m(d-1)}} 
& = 
\lim\limits_{n\to \infty} \frac{ x_n\big(id: \mathring{S}^m_{\infty}W(\T^d) \to L_2(\T^d)\big)}{n^{-m-\frac{1}{2}}(\ln n)^{m(d-1)}}
\\
& = 
\sqrt{ 2m+1}\bigg( \frac{2^d}{(d-1)!}\bigg)^m \, ,
\end{align*}
where $\mathring{S}^m_{\infty}W(\T^d)$ is the subspace of the  Sobolev space of dominating mixed smoothness defined by the extra conditions 
$\int_{-\pi}^\pi f(\bx)\, \dd x_j =0$, $j=1,\ldots,d$.

Finally we shall deal with the behavior of $s$-numbers with respect to the embedding 
$\mathcal{A}^{s,2}_{\mix}(\T^d)\hookrightarrow  H^1(\T^d)$, where $H^1(\T^d)$ is the isotropic Sobolev space of smoothness $1$. 
As a preparation for this we shall supplement the knowledge about the asymptotic behavior of $s$-numbers 
of the embedding   $H^{s,2}_{\mix}(\T^d)\hookrightarrow H^1(\T^d)$,
see, e.g.,  \cite{BuGr04}, \cite{DU13} and \cite{KSU20}, by determining the asymptotic constant. We then show that
\begin{equation}\label{eq-intro-02}
\begin{aligned}
	\lim\limits_{n\to \infty} \frac{a_n\big(id: \mathcal{A}_{\mix}^{s,2}(\T^d) \to H^1(\T^d)\big)}{n^{-s+1}}
	& =  
	\lim\limits_{n\to \infty} \frac{d_n\big(id: \mathcal{A}_{\mix}^{s,2}(\T^d) \to H^1(\T^d)\big)}{n^{-s+1}}
	\\
	& = 
	\bigg(\frac{2s}{2s+1}\bigg)^s(2d)^{s-1}(2S+1)^{(s-1)(d-1)}\, ,
\end{aligned}
\end{equation}
and
\begin{equation}\label{eq-intro-02.5}
\begin{aligned}
\lim\limits_{n\to \infty} \frac{b_n\big(id: \mathcal{A}_{\mix}^{s,2}(\T^d) \to H^1(\T^d)\big)}{n^{-s+ 1/2}} 
&=
\lim\limits_{n\to \infty} \frac{x_n\big(id: \mathcal{A}_{\mix}^{s,2}(\T^d) \to H^1(\T^d)\big)}{n^{-s+ 1/2}}
\\
& = 
\sqrt{ 2s+1}\,(2d)^{s-1}(2S+1)^{(s-1)(d-1)} \,,
\end{aligned}
\end{equation}
where $d\in \N$, $s >1$ and  
\[	
S:=\sum_{k=1}^{\infty}\dfrac{1}{(k^2+1)^{\frac{s}{2(s-1)}}}\, .
\]
These asymptotic formulas are of particular importance with respect to high-dimensional numerical solutions of the Poisson equation, we refer to \cite{BuGr04}. Observe, that there is no logarithmic factor in the asymptotic behavior of these $s$-numbers.
Examining the asymptotic constants in \eqref{eq-intro-01}, \eqref{eq-intro-01.5}
and \eqref{eq-intro-02}, \eqref{eq-intro-02.5}  it becomes clear that there is no simple lifting argument  behind. 
Indeed, this cannot be expected because $H_{\mix}^{1}(\T^d)$ belongs 
to a different scale of spaces than  $H^1(\T^d)$. 
The latter space is not a tensor product space but $H_{\mix}^{1}(\T^d)$ is. So in the 
embedding $\mathcal{A}_{\mix}^{s,r}(\T^d)\hookrightarrow H^1(\T^d)$ we have a break in the scale.
 
Our method consists in a reduction of our original problem  to the 
investigation of $s$-numbers of  diagonal operators between sequence spaces. First we show that the $s$-number under consideration is equal to the corresponding $s$-number of a certain diagonal operator between sequence spaces. The next important step is the rearrangement of the weight (of the diagonal operator) into a non-increasing sequence and the study of its asymptotic behavior. 
Usually this involves some volume estimates and this is the main difficulty here.
Finally, these results are carried over to the function spaces 
and  the desired limits are obtained. Note that this method has been employed in several earlier papers, see e.g., K\"uhn, Mayer, Ullrich \cite{KMU16}
K\"uhn, Sickel, Ullrich \cite{KSU14,KSU15,KSU20} and Cobos, K\"uhn, Sickel \cite{CKS16,CKS19}.
 
It would be also of interest to study the isotropic situation.
In this case 	
the related spaces $\mathcal{A}^{s,r}(\T^d)$ are counterparts of the classical periodic 
Sobolev spaces $H^{s,r} (\T^d)$ on the $d$-dimensional torus. 
The behavior of the $s$-numbers $(s_n)_{n\in \N}$  in dependence on $n$ for the embedding $H^{s,r} (\T^d) \hookrightarrow L_2 (\T^d)$ 
has been investigated at various places, but for investigations in dependence of $n$ and $d$ we refer to \cite{KSU14}.


The paper is organized as follows. 
In Section \ref{vorb} we collect some needed material about the functions spaces and the $s$-numbers under consideration. 
The next Section \ref{weite1} is devoted to the study of 
$s$-numbers of  $id:\, \mathcal{A}_\omega (\T^d) \to X (\T^d)$ for general weights $\omega$ and $ X\in \{\mathcal{A}, L_2,L_\infty\}$.
These results will be used in Section \ref{ex1}, where we deal with the particular family  of weights $\omega_{s,r}$ 
associated to the dominating mixed smoothness.
In the final Section \ref{ex3} we shall determine the asymptotic constant for $s$-numbers of the embedding 
$\mathcal{A}^{s,2}_{\mix}(\T^d) \hookrightarrow H^1(\T^d)$.

\noindent
{\bf Notation.} As usual, $\N$ denotes the natural numbers, $\N_0$ the non-negative integers,
$\zz$ the integers,
$\re$ the real numbers, and $\C$ the complex numbers. By $\tor$ we denote the torus, represented by the interval $[0,2\pi]$, where
the end points of the interval are identified.
The letter $d$ is always reserved for the dimension in $\N^d$, $\Z^d$, $\re^d$, $\C^d$, and $\tor^d$. Vectorial quantities are denoted by boldface
letters and  $x_i$ denotes the $i$th coordinate 
of $\bx \in \R^d$, i.e., $\bx := (x_1,\ldots, x_d)$.
For $0<p\leq \infty$ and $\bx\in \re^d$ we denote $|\bx|_p = \big(\sum_{i=1}^d |x_i|^p\big)^{1/p}$ with the usual modification for $p=\infty$. 
In case $p=2$ we shall  use simply $|\bx|:= |\bx|_2$.
If $X$ and $Y$ are two Banach spaces, $\cL(X,Y)$ denotes the set of  continuous linear operators from $X$ to $Y$. The norm
of an element $x$ in $X$ is denoted by $\|x|X\|$ and the norm of an operator
$A\in \cL(X,Y)$ by $\|A:X\to Y\|$ or $\|A\|$ when there is no 
ambiguity. The symbol $X \hookrightarrow Y$ indicates that there is a
continuous embedding from $X$ into $Y$. 
By $id$ we denote always identity operators.
The symbol $|\Omega|$ stands for the cardinality of the finite set $\Omega$.
If $(a_n)_{n\in \N}$ and $(b_n)_{n\in \N}$ are two sequences of real numbers, the symbol $a_n\sim b_n$ indicates that $\lim_{n\to \infty}\frac{a_n}{b_n}=1$.


\section{$s$-Numbers and function spaces}
\label{vorb}


First we recall  definitions of various $s$-numbers of linear continuous operators and their properties which will play a role in our investigations.


\subsection{$s$-Numbers}


Let $X$, $Y$ be Banach spaces and $T$ be a continuous linear operator from $X$ to $Y$. 
Then the $n$th approximation number of $T$ is defined as
$$ a_n(T):=\inf\big\{\|T-A\|: \ A\in \mathcal L(X,Y),\ \ \text{rank} (A)<n\big\}\, , \qquad n \in \N\, . $$
The $n$th Kolmogorov number of the linear operator $T $ is defined as 
\begin{equation*}
d_n(T)= \inf_{L_{n-1}}\sup_{\|x|X\|\leq 1}\inf_{y\in L_{n-1}}\|Tx-y|Y\|, 
\end{equation*}
where the left-most  infimum is  taken over all $(n-1)$-dimensional subspaces $L_{n-1}$ in $Y$.
\\
The $n$th Weyl number of $T$ is defined as 
	$$ x_n(T):=\sup\big\{a_n(TA):\ A\in \mathcal L(\ell_2,X),\ \|A\|\leq 1\big\}\, , \qquad n \in \N\, .$$
The $n$th Bernstein number of $T$ is defined as
	$$ b_n(T)= \sup_{L_n}\inf_{x\in L_n
			\atop 0 < \|x|X\| \leq 1}  \frac{\|Tx\,|\,Y\|}{\|x\, |\, X\|} ,$$
	where the supremum is taken over all $n$-dimensional subspaces $L_n$ in $X$.

There is a huge amount of literature about  $s$-numbers and related widths. The concept of widths came from the Soviet
School in Approximation Theory. It was used to measure the approximation error of sets in normed spaces. Kolmogorov widths are the oldest widths suggested by Kolmogorov in \cite{Kol36}. Linear and
Bernstein widths were introduced by Tikhomirov in \cite{Tikho60}.  In order to approximate general operators between Banach spaces,  Pietsch defined the notion  $s$-number in  \cite{Pie74}. Kolmogorov, linear, and Bernstein widths are Kolmogorov, approximation, and Bernstein numbers of identity operators, respectively.
There is one more point we would like to mention.
Approximation, Kolmogorov,  Bernstein and Weyl numbers are $s$-numbers in the sense of Pietsch \cite{Pie74}, \cite[Section 11.1]{Pie80B}.
However, later  in \cite[Section 2.2]{Pie87B}, \cite{Pie08} Pietsch modified the definition of $s$-numbers. Bernstein numbers do not satisfy these new conditions, but the other still do. 
All these numbers have its own history and their different roles in approximation and operator theory.
Good references are the monographs by Pietsch \cite{Pie80B,Pie87B,{Pie07B}} and of Pinkus \cite{Pin85B}.

\subsection*{Some properties}

Many times we shall employ the ideal property, i.e., if $s=(s_n)_{n \in \N}$ is one of the above sequences, then 
\be\label{ws-04}
s_n (RST)\le \|R\|\, s_n (S)\, \|T\|\, , \qquad n \in \N\, ,
\ee
holds, see \cite{Pie74}, \cite[Section 11.1]{Pie80B}.
Here $R,S,T$ are linear operators satisfying
$T \in \mathcal L(X_0,X)$, $S \in \mathcal L(X,Y)$, and
$R \in \mathcal L(Y,Y_0)$ for certain Banach spaces $X_0,X,Y_0,Y$.
If $T$ is an compact operator between two Hilbert spaces, then
\begin{equation}\label{eq-eq}
a_n(T)=d_n(T)=x_n(T)=b_n(T)
\end{equation}
and they coincide with  the singular numbers of $T$. 
In general, the approximation numbers represent the largest $s$-numbers.
For compact operators $T$ it holds
\beqq
 b_n (T) \le  d_n (T) \le a_n (T)\, , \qquad n \in \N\, ,
\eeqq
see, e.g., \cite[Theorems 8.1, 8.2]{Pie74}. 
There is no general inequality relating Weyl and Bernstein numbers, Weyl and Kolmogorov numbers, see, e.g., \cite{Pie08}.
However, if $T:~X \to Y$, where $Y$ is a Hilbert space, then there is a relation between them. 
For any $A: \ell_2\to X$ with $\|A\|\leq 1$ we have $TA:~\ell_2 \to Y$ and by \eqref{eq-eq} we obtain
\begin{equation*}
	a_n(TA)=b_n(TA)\leq \|A\| \cdot b_n( T:X\to Y)\leq b_n( T:X\to Y)\, .
\end{equation*}
It follows immediately from the definition of the Weyl numbers that  
\be\label{weylb}
x_n(T)\leq b_n(T)
\ee
in this special situation.
If $Y$ is a Hilbert space we also have 
\beqq
a_n(T)=d_n(T), 
\eeqq
see, e.g., \cite[Proposition 11.6.2]{Pie80B}.

As it will be seen below, Bernstein numbers and Weyl numbers will behave similarly in our context and in the same way 
approximation numbers and Kolmogorov numbers.  
Therefore we will use the following convention.

\noindent
{\bf Convention:} The  notation $v_n(T)$ will be used instead of  $x_n(T)$ and $b_n(T)$.
The notation  $u_n(T)$  will be used instead of  $a_n(T)$ and $d_n(T)$. 
This means if we say $v_n$ has properties $x,y,z$ then  both, 
Bernstein numbers and Weyl numbers have these properties etc..


\subsection{Function spaces}


Let $\T^d$ be the $d-$dimensional torus. We equip $\T^d$ with the probability measure $(2\pi)^{-d}\dd \bx$.  It implies that $\T^d$ has volume 1  and this is needed for the theory developed in \cite{CKS16}. 
Furthermore it has the advantage that several related embeddings will have operator norm equal to $1$. 
For a function $f\in L_1(\T^d)$, its Fourier coefficients are defined as
\begin{equation*}
	\hat{f}(\bk):=(2\pi)^{-d}\int_{\T^d} f(\bx) e^{-{\rm i} \bk \bx} \dd \bx,\qquad \bk\in\Z^d.
\end{equation*}
Note, that the system $\{e^{{\rm i}\bk \bx}:\ \bk \in \Z^d\}$ is an orthonormal basis in $L_2(\T^d)$. Hence, it holds for any $f\in L_2(\T^d)$ that
\begin{equation*}
	\|f|L_2(\T^d)\|^2= (2\pi)^{-d}\int_{\T^d}|f(\bx)|^2  \dd \bx = \sum_{\bk \in \Z^d} |\hat{f}(\bk)|^2 \,.
\end{equation*}

Let $\omega=(\omega(\bk))_{\bk\in \Z^d}$  be a sequence of positive numbers. In what follows we will call such a sequence a weight.
We introduce the weighted Wiener algebra 
$\mathcal{A}_\omega(\T^d)$ as the collection of all  functions $f\in L_1(\T^d)$ such that 
\[
\|f|\mathcal{A}_\omega(\T^d)\|:= \sum_{\bk\in \Z^d} \omega(\bk) |\hat{f}(\bk)|< \infty \, .
\]
In case $\omega(\bk)= 1$ for all $\bk\in \Z^d$ we get back the classical  Wiener algebra $\mathcal{A}(\T^d)$. 
Let us add a comment regarding the name of $\mathcal{A}_\omega(\T^d)$ to avoid misunderstanding.
We do not claim here that $\mathcal{A}_\omega(\T^d)$ is an algebra.
In general this is not true. But those questions are out of scope here.

As another scale of function spaces we introduce some weighted Hilbert spaces built on $L_2 (\T^d)$.
The class $F_\omega(\T^d)$  is the collection of all integrable functions  $f\in L_1(\T^d)$  such that 
\begin{equation*}
	\| f|F_\omega(\T^d)\| := \Bigg(\sum_{\bk\in \Z^d} \omega(\bk)^2|\hat{f}(\bk)|^2\Bigg)^{1/2}<\infty\,.
\end{equation*}
For later use we remark that
\be
\label{ws-40}
\|\, f\, |F_\omega (\T^d)\| \le \|\, f\, |\mathcal{A}_\omega (\T^d)\|\, ,  
\ee
or with other words, the norm of $id: ~ \mathcal{A}_\omega (\T^d) \to 
{F}_\omega (\T^d)$ is equal to $1$.
In what follows we suppose 
\be \label{ws-07}
\lim_{|\bk|_1 \to \infty}  |\omega(\bk)| = \infty \, .
\ee
Then an important role will be played by the non-increasing rearrangement of the sequence $(1/\omega(\bk))_{\bk\in \Z^d}$, 
denoted by  $(\sigma_n)_{n\in \N}$.
Observe that \eqref{ws-07} implies on the one side the compactness of the embedding  $F_\omega(\T^d) \hookrightarrow  L_2(\T^d)$
and on the other side $\lim_{n\to \infty}\, \sigma_n =0$.

\begin{lemma}\label{sing}
	Let $(\omega(\bk))_{\bk\in \Z^d}$  be a sequence of positive  numbers such that \eqref{ws-07} is satisfied.
	Let $s_n \in \{x_n,b_n,d_n,a_n\}$.
	Then we have 
	\begin{equation*}
		s_n\big(id: F_\omega(\T^d) \to L_2(\T^d)\big)=\sigma_n,\qquad n\in \N\, .
	\end{equation*}
\end{lemma}

This lemma represents nothing but a convenient reformulation of a classical result on the 
behavior of approximation numbers of certain diagonal operators, see e.g. Pietsch \cite[Theorem 11.3.2.]{Pie80B} or
Pinkus \cite[Theorem IV.2.2]{Pin85B}.
Comments on the history may be found in Pietsch \cite[6.2.1.3]{Pie07B}.


\section{$s$-Numbers of embeddings of weighted Wiener algebras}
\label{weite1}


In this section,  we will consider  two different cases, namely 
$s_n\big(id: \mathcal{A}_\omega(\T^d) \to \mathcal{A}(\T^d)\big)$
and 
$s_n\big(id: \mathcal{A}_\omega(\T^d) \to L_2(\T^d)\big)$, 
where $s_n \in \{a_n, b_n, d_n, x_n\}$.
In addition a supplement to the case $id: \mathcal{A}_\omega(\T^d) \to L_\infty(\T^d)$
will be given.


\subsection{$s$-Numbers of $id: ~ \mathcal{A}_\omega(\T^d) \to \mathcal{A}(\T^d)$}


The main result in this section will be the 
useful identity \eqref{ws-12b} below, very easy to prove.
Before calculating
$s_n\big(id: \mathcal{A}_\omega(\T^d) \to \mathcal{A}(\T^d)\big)$ we have the following.

\begin{theorem}\label{satz1}
	Let $\omega = (\omega(\bk))_{\bk\in \Z^d}$ be a weight satisfying 
	\eqref{ws-07}. Let $s_n \in \{x_n, b_n, d_n, a_n\}$. \\
	{\rm (i)} Then we have 
	\be\label{ws-08}
	s_n\big(id: \mathcal{A}_\omega(\T^d) \to \mathcal{A}(\T^d)\big)
	=\sigma_n,\qquad n\in \N\, .
	\ee
	{\rm (ii)} Let $n \in \N$.
	Let $\Lambda \subset \Z^d$ be a set with the two properties:
	\begin{itemize}
		\item The cardinality $|\Lambda|$ of $\Lambda$ is $n-1$.
		\item For any $\bk\in \Lambda$ and any $\bl \not\in \Lambda$ it holds $\omega (\bk)\le \omega (\bl)$.
	\end{itemize}
	For $f\in L_1 (\T^d)$ we define 
	\[
	S_\Lambda f (\bx):= \sum_{\bk \in \Lambda} \hat{f}(\bk) \, e^{{\rm i} \bk \bx}\,, \qquad \bx\in \T^d\, .
	\]
	Then 
	\beqq
	a_n\big(id: \mathcal{A}_\omega(\T^d) \to \mathcal{A}(\T^d)\big) 
	= \sup_{\|f\, |\mathcal{A}_\omega(\T^d)\|\le 1}\, \| \, f-S_\Lambda  f\, |\mathcal{A}(\T^d)\|\, .
	\eeqq
\end{theorem}
Observe that the left-hand side in the last equation is independent of the
choice of set $\Lambda$, therefore the  
right-hand side  is independent of $\Lambda$ as well.

\begin{proof}
	{\em Step 1.} 
	Proof of (i). We consider the following commutative diagram
	
	\begin{equation*}
		\begin{CD}
			\mathcal{A}_\omega(\T^d)  @ > id >> \mathcal{A}(\T^d) \\
			@VV A V @AA B A\\
			\ell_1(\Z^d) @ > D_\omega >> \ell_1(\Z^d) \, ,
		\end{CD}
	\end{equation*}
	where the linear operators $A:\, \mathcal{A}_\omega(\T^d) \to \ell_1 (\Z^d)$, 	$D_\omega:\, \ell_1 (\Z^d) \to \ell_1 (\Z^d)$, and $B:\, \ell_1 (\Z^d) \to \mathcal{A}(\T^d)$ 
	are defined as 
\begin{equation}	\label{ws-12}
	\begin{aligned}
Af & : =  (\omega(\bk)\hat{f}(\bk))_{\bk\in \Z^d}\, ,  
\\
D_\omega\xi & :=   (\xi(\bk)/\omega(\bk))_{\bk\in \Z^d}\,  , \qquad \xi=(\xi(\bk))_{\bk\in \Z^d}
\\
(B\xi)(\bx)& :=  \sum_{\bk\in \Z^d} \xi(\bk)\,  e^{{\rm i}\bk \bx}\, , \qquad \bx \in \T^d\, .
	\end{aligned}
\end{equation}

	It is obvious that $\|A\|=\|B\|=1$.
	From the ideal property \eqref{ws-04} and the identity $id = B\,D_\omega \, A$ it follows
	\[
	s_n\big(id: \mathcal{A}_\omega(\T^d) \to \mathcal{A}(\T^d)\big) \le s_n \big(D_\omega: \ell_1 (\Z^d) \to \ell_1(\Z^d)\big) \, .
	\]
	Theorem 7.1 in \cite{Pie74}, see also Proposition 2.9.5 in \cite{Pie87B}, yields
	\[
	s_n\big(D_\omega: \ell_1 (\Z^d) \to \ell_1(\Z^d)\big) = \sigma_n\, .
	\]
	This proves the estimate from above. 
	Now we employ the same type of arguments with respect to the diagram

	\begin{equation*}
		\begin{CD}
			\ell_1(\Z^d) @ > D_\omega >> \ell_1(\Z^d)
			\\
			@VV A^{-1} V @AA B^{-1} A\\
			\mathcal{A}_\omega(\T^d)  @ > id >> \mathcal{A}(\T^d)\,.
		\end{CD}
	\end{equation*}

	It is easy to see that the operators $A$ and $B$ are invertible and  that $\|A^{-1}\|=\|B^{-1}\|=1$.
	As above we conclude 
	\[
	\sigma_n = s_n\big(D_\omega: \ell_1 (\Z^d) \to \ell_1(\Z^d)\big) 
	\le 
	s_n\big(id: \mathcal{A}_\omega(\T^d) \to \mathcal{A}(\T^d)\big) \, ,
	\]
	which completes the proof of (i).
	\\
	{\em Step 2.} Proof of (ii). By definition of the sequence $(\sigma_n)_{n\in \N}$
	we have 
\begin{align*}
	\| \, f- S_\Lambda f \, |\mathcal{A}(\T^d)\| & = 
\sum_{\bk \in \Z^d\setminus \Lambda} |\hat{f}(\bk)| \le 
\sup_{\bl  \in \Z^d\setminus \Lambda}\, \frac{1}{\omega (\bl)}\, \Bigg(
\sum_{\bk \in \Z^d\setminus \Lambda} |\omega (\bk)\, \hat{f}(\bk)|\Bigg)
\le   \sigma_n \, \|\, f\, |A_\omega (\T^d)\|\, 
\end{align*}
	which implies  
	\[
	\| \, id - S_\Lambda  \|\le \sigma_n =
	a_n\big(id: \mathcal{A}_\omega(\T^d) \to \mathcal{A}(\T^d)\big) \, .
	\]
	From  the definition of $a_n$ we know $a_n(id) \le \| \, id - S_\Lambda \|$.
	This proves (ii).
\end{proof}

\begin{cor}\label{cor1}
	Let $\omega = (\omega(\bk))_{\bk\in \Z^d}$ be a weight satisfying 
	\eqref{ws-07}. Let $s_n \in \{x_n, b_n, d_n, a_n\}$. 
	Then, for all $n\in \N$, we have 
	\be\label{ws-12b}
	s_n\big(id: \mathcal{A}_\omega(\T^d) \to \mathcal{A}(\T^d)\big)  = s_n\big(id: F_\omega(\T^d) \to L_2(\T^d)\big) \, .
	\ee
	Furthermore, in both cases the operator $S_\Lambda$, defined in Theorem \ref{satz1}(ii), realizes the optimal approximation, 
	i.e.,
	\[
	\| \, id - S_\Lambda   :\mathcal{A}_\omega (\T^d) \to \mathcal{A}(\T^d)\| =  a_n\big(id: \mathcal{A}_\omega(\T^d) \to \mathcal{A}(\T^d)\big) 
	\]
	and 
	\[
	\| \, id - S_\Lambda   :F_\omega (\T^d) \to L_2(\T^d)\| =  a_n\big(id: F_\omega(\T^d) \to L_2(\T^d)\big) \, .
	\]
\end{cor}

\begin{proof}
	Lemma \ref{sing} and Theorem \ref{satz1} yield the first part of the claim.
	The second part is a consequence of Theorem \ref{satz1}(ii) and Remark 4.2 in \cite{KSU14}.
\end{proof}


\subsection{$s$-Numbers of $id: ~ \mathcal{A}_\omega(\T^d) \to L_2(\T^d)$}


Now we turn to the  $s$-numbers of the embedding $\mathcal{A}_\omega(\T^d) \hookrightarrow L_2(\T^d)$. Here we need our convention.
Weyl and Bernstein numbers  have a different asymptotic behavior compared to the asymptotic behavior of 
approximation and Kolmogorov numbers.

\begin{theorem}
	\label{thm:bw} 
	Let $\omega = (\omega(\bk))_{\bk\in \Z^d}$ be a weight satisfying 
	\eqref{ws-07}.
	\\
	{\rm (i)} In case of Weyl and Bernstein numbers  we have
	\begin{equation*}
		v_n\big(id: \mathcal{A}_\omega(\T^d) \to L_2(\T^d)\big)= \Bigg(\sum_{k=1}^n \sigma_k^{-2}\Bigg)^{-1/2}, \qquad n\in \N\,.  
	\end{equation*}
	{\rm (ii)}
	For approximation and Kolmogorov numbers it holds	
	\begin{equation*} 
		\begin{aligned}
			u_n\big(id: \mathcal{A}_\omega(\T^d) \to L_2(\T^d)\big)
			= 
			\sup_{h\geq n} \Bigg(\frac{h-n+1}{\sum_{k=1}^h\sigma_k^{-2}}\Bigg)^{1/2}\,,\qquad n\in \N.  
		\end{aligned}
	\end{equation*} 
\end{theorem}

\begin{proof} The proof is quite similar to the proof of Theorem \ref{satz1}.
	Again let $s_n \in \{x_n, b_n , d_n, a_n\}$.
	We consider the following diagram
	\begin{equation*}
		\begin{CD}
			\mathcal{A}_\omega(\T^d)  @ > id >> L_2(\T^d) \\
			@VV A V @AA B A\\
			\ell_1(\Z^d) @ > D_\omega >> \ell_2(\Z^d) \, 
		\end{CD}
	\end{equation*}
	where $A, B$ and  $D_\omega $
	are defined as in \eqref{ws-12}.
	Again we have  $\|A\|=\|B\|=1$. By the ideal property of these numbers, 
	see \eqref{ws-04}, we obtain
	\begin{equation*}
		\begin{aligned}
			s_n\big(id: \mathcal{A}_\omega (\T^d)\to L_2(\T^d)\big)
			&\leq \, s_n\big(D_\omega: \ell_1(\Z^d)\to \ell_2(\Z^d)\big)\,.
		\end{aligned}
	\end{equation*}
	Let $D_\sigma$ be the diagonal operator from $\ell_1(\N)\to \ell_2(\N)$ defined by 
	$D_\sigma \xi :=(\sigma_k \, \xi_k)_{k\in \N}$. Since $(\sigma_n)_{n\in \N}$ is the non-increasing rearrangement of 
	$(1/\omega(\bk))_{\bk\in \Z^d}$ we obtain 
	\begin{equation*}
		\begin{aligned}
			s_n\big(D_\omega: \ell_1(\Z^d)\to \ell_2(\Z^d)\big)
			& =\,  	s_n\big(D_\sigma : \ell_1(\N)\to \ell_2(\N)\big) \, ,
		\end{aligned}
	\end{equation*}
	which leads to 
	\begin{equation}\label{eq:<=}
		\begin{aligned}
			s_n\big(id: \mathcal{A}_\omega (\T^d)\to L_2(\T^d)\big) \leq 
			s_n\big(D_\sigma : \ell_1(\N)\to \ell_2(\N)\big) .
		\end{aligned}
	\end{equation}
	To prove the reverse direction we consider the modified diagram
	\begin{equation*}
	\begin{CD}
		\ell_1(\Z^d) @ > D_\omega >> \ell_2(\Z^d)
		\\
		@VV A^{-1} V @AA B^{-1} A\\
		\mathcal{A}_\omega(\T^d)  @ > id >> L_2(\T^d)\,.
	\end{CD}
\end{equation*}
	By the same argument 
	as used above we get the reverse inequality  of \eqref{eq:<=}. 
	Consequently we find that
	\begin{equation*} 
		\begin{aligned}
			s_n\big(id: \mathcal{A}_\omega (\T^d)\to L_2(\T^d)\big) = 
			s_n\big(D_\sigma : \ell_1(\N)\to \ell_2(\N)\big) .
		\end{aligned}
	\end{equation*}
	Finally, from 
	\begin{equation*} 
		\begin{aligned}
			v_n\big(D_\sigma : \ell_1(\N)\to \ell_2(\N)\big) 
			=  \Bigg(\sum_{k=1}^{n}\sigma_k^{-2}\Bigg)^{-1/2}\,
		\end{aligned}
	\end{equation*}
	see \cite{Pie74,Gal91} and
	$$
	u_n\big(D_\sigma : \ell_1(\N)\to \ell_2(\N)\big)
	=
	\sup_{h\geq n} \Bigg(\frac{h-n+1}{\sum_{k=1}^h\sigma_k^{-2}}\Bigg)^{1/2},
	$$ \cite[Theorem 11.11.7]{Pie80B}, we obtain the desired results.
\end{proof}

\begin{theorem}
\label{thm-general} 
Let $s>0$,  $\beta\geq 0$ and $\omega$  be a weight satisfying \eqref{ws-07}. Assume that there exists a  number 
$C$ such that
\begin{equation}\label{eq-assumption}
 \lim_{n \to \infty} \, \frac{s_{n} \big(id: \, F_{\omega} (\T^d) \to L_2(\T^d)\big)}{n^{-s}(\ln n)^{\beta} \, } = C\, .
\end{equation}
	{\rm (i)} In case of Bernstein or Weyl numbers we have 
	\begin{equation*}
		\lim\limits_{n\to \infty} \frac{ v_n\big(id: \mathcal{A}_{\omega}(\T^d) \to L_2(\T^d)\big)}{n^{-s-\frac{1}{2}}(\ln n)^{\beta}}= 
		\sqrt{ 2s+1}\, C \,.
	\end{equation*}
	{\rm (ii)}
	For approximation or Kolmogorov numbers it holds
	\begin{equation*}
		\lim\limits_{n\to \infty} \frac{ u_n\big(id: \mathcal{A}_{\omega}(\T^d) \to L_2(\T^d)\big)}{n^{-s}(\ln n)^{\beta}}
		= \bigg(\frac{2s}{2s+1}\bigg)^s C\, .
	\end{equation*}
\end{theorem}

 \begin{proof} {\it Step 1.} In this proof we use the following facts which are not difficult to verify. Let $(a_n)_{n\in \N}$, $(b_n)_{n\in \N}$ be sequences of positive numbers and $a_n \sim b_n$.  
 \begin{itemize}
 	\item If  $(a_n)_{n\in \N}$ is bounded, then
 	\begin{equation}\label{eq-fact1}
 		\sup_{h\geq n} a_h \sim \sup_{h\geq n}b_h .
 	\end{equation}
 \item If $\lim_{n\to \infty}a_n=\infty$ then for  $k_1, k_2\in \N$ it holds
 \begin{equation}\label{eq-fact2}
 	\sum_{k=k_1}^n a_k \sim  \sum_{k=k_2}^n b_k .
 \end{equation}
 \end{itemize}	
{\it Step 2.} Proof of (i).	From Theorem \ref{thm:bw} we have
	\begin{equation*} 
		\begin{aligned}
			v_n^{-2} := 	v_n\big(id: \mathcal{A}_{\omega}(\T^d) \to L_2(\T^d)\big)^{-2}=  \sum_{k=1}^n\sigma_k^{-2} .
		\end{aligned}
	\end{equation*}
By Lemma \ref{sing} it follows  $\sigma_k = s_k\big(id: F_{\omega}(\T^d) \to L_2(\T^d)\big)$. 
	Assumption \eqref{eq-assumption} indicates that 
	\begin{equation*}
\sigma_n^{-1} \sim  C^{-1}n^{s} (\ln n)^{-\beta}
	\end{equation*}
	which by \eqref{eq-fact2} implies
	\begin{equation*}
		v_n^{-2} 
		\sim
		 C^{-2} \sum_{k=2}^n k^{2s} (\ln k)^{-2\beta}\,. 
	\end{equation*} 
	Observe that  $f(t)=t^{2s}(\ln t)^{-2\beta}$ is an increasing function if  $t\geq t_0>0$ for some  $t_0= t_0 (s,\beta) $.
Then we have
	\begin{equation*}
		\begin{aligned}
			\sum_{k=2}^n \frac{k^{2s}}{  (\ln k)^{2\beta}}
			&
			\sim  \int_{2}^{n} \frac{ t^{2s}}{ (\ln t)^{2\beta}} \dd t
 \overset{t=yn}{=} \frac{n^{2s+1}}{(\ln n)^{2\beta}}  \int_{\frac{2}{n}}^1 y^{2s}\bigg(\frac{\ln n}{\ln(y n)}\bigg)^{2\beta}\dd y\,.
		\end{aligned}
	\end{equation*}
From this and  Lemma \ref{lem:auxi-1} below we derive 
	\begin{equation}\label{eq-fact3}
\frac{	v_n^{-2} }{n^{2s+1}(\ln n)^{-2\beta}} \sim C^{-2} 
\int_{\frac{2}{n}}^1 y^{2s}\bigg(\frac{\ln n}{\ln(y n)}\bigg)^{2\beta}\dd y \sim  \frac{C^{-2}}{2s+1}. 
	\end{equation}
	This proves (i).\\
	{\it Step 3.} Proof of (ii).
By using Theorem \ref{thm:bw},  \eqref{eq-fact1} and  \eqref{eq-fact3} we obtain
	\begin{equation*} 
		\begin{aligned}
			u_n^2
			& :=	 u_n\big(id: \mathcal{A}_{\omega}(\T^d) \to L_2(\T^d)\big)^2
		= 
			\sup_{h\geq n} \Bigg(\frac{h-n+1}{\sum_{k=1}^h\sigma_k^{-2}}\Bigg) \\
			&
			=\sup_{h\geq n} (h-n+1)v_h^2 \sim C^2(2s+1)\sup_{h\geq n }  \frac{(h-n)(\ln h)^{2\beta}}{h^{2s+1}}.
		\end{aligned}
	\end{equation*}
It follows that
\begin{equation}\label{eq-k01}
\frac{u_n^2}{n^{-2s}(\ln n)^{2\beta}} \sim C^2(2s+1) \sup_{h\geq n }  \frac{(h-n)(\ln h)^{2\beta}}{h^{2s+1}} \cdot \frac{n^{2s}}{(\ln n)^{2\beta}}.
\end{equation}
Considering the function 
$
	g(h) := \frac{h-n}{h^{2s+1} } (\ln h)^{2\beta}\,,   h\in [n,\infty)\,
$
we have
 $g'(h)\leq 0$ is equivalent to
	\begin{equation*} 
	f(h):=	\big[-2sh+n(2s+1)\big]\ln h  + \big(h-n
	\big)2\beta  \leq 0.
	\end{equation*} 
	It is easily seen that
$  
	f\big( n+\frac{n}{s} \big)     <0\, 
$
	for $n\geq n_0$ depending on $s$ and $\beta$. Then $g(h)$ attains it supremum in $\big[n,n+\frac{n}{s}\big]$. Moreover, the function $g_1(h):= \frac{h-n}{h^{2s+1}}$, $h \in [n,\infty)$, attains its maximum at 
$h= n+\frac{n}{2s}$. 	Consequently, we find 
	\begin{equation} \label{eq-k02}
		\begin{aligned}
\sup_{h\geq n }  \frac{(h-n)(\ln h)^{2\beta}}{h^{2s+1}} \cdot \frac{n^{2s}}{(\ln n)^{2\beta}}
			&	\leq   
			\frac{1 }{2s\big(1+\frac{1}{2s}\big)^{2s+1}} \bigg(\frac{\ln(n+\frac{n}{s})}{\ln n}\bigg)^{2\beta}\,.
		\end{aligned}
	\end{equation} 
Choosing $h=  n+\lfloor\frac{n}{2s}\rfloor $, where  $\lfloor a \rfloor$ denotes the greatest integer not larger than $a$,   we conclude  
	\begin{equation}\label{eq-k03} 
	\begin{aligned}
		\sup_{h\geq n }  \frac{(h-n)(\ln h)^{2\beta}}{h^{2s+1}} \cdot \frac{n^{2s}}{(\ln n)^{2\beta}}
		&	\geq   
		\frac{\lfloor \frac{n}{2s}\rfloor \, n^{2s}}{\big(n+\lfloor\frac{n}{2s}\rfloor\big)^{2s+1}} \bigg(\frac{\ln(n+\lfloor\frac{n}{2s}\rfloor)}{\ln n}\bigg)^{2\beta}\sim 	\frac{1 }{2s\big(1+\frac{1}{2s}\big)^{2s+1}}\,.
	\end{aligned}
\end{equation} 
From \eqref{eq-k01}, \eqref{eq-k02}, and \eqref{eq-k03} the assertion follows.
\end{proof}

\begin{lemma}\label{lem:auxi-1} Let $s>0$, $\beta\geq 0$, and $a > 1$. Then we have
	\begin{equation*}
		\lim\limits_{n\to \infty} \, \int_{\frac{a}{n}}^1 y^{s}\bigg(\frac{\ln n}{\ln(yn)}\bigg)^{\beta}\dd y =\frac{1}{s+1}.
	\end{equation*}
\end{lemma}

\begin{proof} 
	We consider the sequence of functions 
	\[
	f_n(y) = y^{s}\, \bigg(\frac{\ln n}{\ln(yn)}\bigg)^{\beta} \Chi_{[a/n,1]}\,, \qquad y \in (0,1)\, ,\ n>a,
	\]
where $\Chi_{[a/n,1]}$ is the characteristic function of $[a/n,1]$.	It is clear that this sequence converges pointwise to $f(y)=y^s$ on $(0,1)$.
	Now we turn to the existence of a common majorant.
	Since case $\beta =0$ is obvious, we concentrate on $\beta >0$.
	The  derivative of $f_n$ on $(a/n,1)$ is given by
	\[
	f_n'(y) = y^{s-1}\, \bigg(\frac{\ln n}{\ln(yn)}\bigg)^{\beta} \bigg[s-\frac{\beta}{\ln (ny)}\bigg]\, .
	\]
	This function has at most one sign change in $(a/n,1)$.
	Hence, the maximal value of $f_n$ with respect to the interval $[a/n,1]$ is attained either at $a/n$, $1$ or
	$e^{\beta/s}/n$.
	This implies 
	\[
	\max_{a/n \le y \le 1} \, |f_n(y)|= \max \bigg\{\Big(\frac an\Big)^s \Big(\frac{\ln n}{\ln a} \Big)^\beta, 1, 
	\frac{s}{\beta} \, e^{\beta}\,  \frac{\ln^\beta n}{n^s}\bigg\} \, .
	\]
	Summarizing we found that there exists a constant $C_{a,s,\beta}>0 $ such that
$
	\sup_{y\in (0,1)} |f_n(y)|  \leq C_{a,s,\beta}  
$
	holds. Hence, the desired result follows from  Lebesgue's dominated convergence theorem.
\end{proof}


\subsection{Approximation and Kolmogorov numbers of 
$id: ~ \mathcal{A}_\omega(\T^d) \to L_\infty(\T^d)$}


The Hilbert space $F_\omega (\T^d)$ is continuously embedded into $\mathcal{A}(\T^d)$ if and only if 
$\sum_{\bk \in \Z^d} \, \omega (\bk)^{-2}$ is finite, see \cite{CKS16}.
Under this restriction
it has been proved in  \cite{CKS16}  that the approximation numbers of embeddings of the 
Hilbert spaces $F_\omega (\T^d)$ have the following property:
\begin{equation}\label{ws-61}
a_n\big(id: F_\omega(\T^d) \to L_\infty(\T^d)\big) = 
a_n\big(id: F_\omega(\T^d) \to \mathcal{A}(\T^d)\big)\, , \qquad n \in \N\, .
\end{equation}
In this section we would like to understand whether a similar property is true for 
the  approximation numbers of embeddings of the weighted Wiener algebras 
$\mathcal{A}_\omega (\T^d)$.
As a  preparation for later considerations  we mention  the following partial result.

 \begin{lemma}\label{thm:ad-infty} 
 Let $\omega = (\omega(\bk))_{\bk\in \Z^d}$ be a weight satisfying 
 \eqref{ws-07}. Then we have
\begin{equation*} 
\begin{aligned}
\sup_{h\geq n} \Bigg(\frac{h-n+1}{\sum_{k=1}^h\sigma_k^{-2}}\Bigg)^{1/2}	\leq 
u_n\big(id: \mathcal{A}_\omega(\T^d) \to L_\infty(\T^d)\big)
\leq  \sigma_n \,,\qquad n\in \N.  
 \end{aligned}
\end{equation*}
\end{lemma}

\begin{proof} 
We consider the following chain of embeddings
\begin{equation*}
\mathcal{A}_\omega(\T^d) \hookrightarrow \mathcal{A}(\T^d) \hookrightarrow L_\infty(\T^d) \hookrightarrow L_2(\T^d).
\end{equation*}
Then the lower bound follows from  Theorem \ref{thm:bw} and 
\begin{equation*} 
\begin{aligned} 
u_n\big(id: \mathcal{A}_\omega(\T^d) \to L_2(\T^d)\big)
	&
	\leq
	\big\|id: L_\infty(\T^d) \to  L_2(\T^d)  \big\|\cdot u_n\big(id: \mathcal{A}_\omega(\T^d) \to L_\infty(\T^d)\big)
	\\
	& 
	= 
	u_n\big(id: \mathcal{A}_\omega(\T^d) \to L_\infty(\T^d)\big)\,.
 \end{aligned}
\end{equation*}
For the upper bound we have
\begin{equation*} 
\begin{aligned}
	u_n\big(id: \mathcal{A}_\omega(\T^d) \to L_\infty(\T^d)\big)
	&
	\leq
	\big\|id: \mathcal{A}(\T^d) \to L_\infty(\T^d)  \big\|\cdot u_n\big(id: \mathcal{A}_\omega(\T^d) \to \mathcal{A}(\T^d)\big)
	\\
	& 
	= 
u_n\big(id: \mathcal{A}_\omega(\T^d) \to \mathcal{A}(\T^d)\big)\,.
 \end{aligned}
\end{equation*}
Finally we use \eqref{ws-08} to get the desired result.
\end{proof}

\begin{remark}
 \rm
In some special situations the quantities 
$
\sup_{h\geq n} \Big(\frac{h-n+1}{\sum_{k=1}^h\sigma_k^{-2}}\Big)^{1/2}$ and $\sigma_n
$
will be of the same order. 
In such a situation we will not be able to show the existence of an asymptotic constant. However, we will get information about 
the correct  order of the decay of the $(u_n)_{n\in \N}$.
\end{remark}

All the results we obtained so far have the disadvantage that the sequence $(\sigma_n )_{n\in \N}$ shows up.
Below we shall discuss a  special family of weights $\omega_{s,r}$ and the corresponding spaces $\mathcal{A}_{\omega_{s,r}} (\T^d)$ 
and $F_{\omega_{s,r}} (\T^d)$. Then we will be able to remove this dependence. 
The nonincreasing rearrangement $(\sigma_n)_{n\in \N}$  of our weight will be replaced by explicit 
quantities in $n,s,d$ and $r$.


\section{Function spaces of dominating mixed smoothness}
\label{ex1}


In this section we shall deal with the family of weights 
\begin{align*} 
\omega_{s,r}(\bk) & :=  \prod_{i=1}^d\big(1+|k_i|^r\big)^{s/r}\, , \qquad 0 < r < \infty\, , \\
\omega_{s,\infty}(\bk) & :=  \prod_{i=1}^d \max (1,|k_i|)^{s} , 
\end{align*}
$\bk \in \Z^d$.
Here the parameter $s$ satisfies  $0<s<\infty$. 
We shall use the notation  $\mathcal{A}_{\mix}^{s,r}(\T^d):=\mathcal{A}_{\omega_{s,r}}(\T^d)$
and ${H}_{\mix}^{s,r}(\T^d):={F}_{\omega_{s,r}}(\T^d)$, respectively.
In both cases, for different $r$, we obtain the same sets of functions.
A change of the parameter $r$ leads to a change of the quasinorm only.

The classes ${H}_{\mix}^{s,r}(\T^d)$ are well-known in approximation theory. 
They are called periodic Sobolev spaces of dominating mixed smoothness.
The space ${H}_{\mix}^{s,r}(\T^d)$  is just a tensor product space of the  univariate 
Sobolev spaces  ${H}^{s,r}(\T)$, i.e., 
\[
{H}_{\mix}^{s,r}(\T^d) = {H}^{s,r}(\T) \otimes \ldots \otimes {H}^{s,r}(\T)
\]
(to be understood as the iterated tensor product of $d$ Hilbert spaces).
Let $s = m\in \N$. We define  the space $H^m_{\mix}(\T^d)$ to be the collection of all functions 
$f \in L_2 (\T^d)$ such that all distributional derivatives $D^\balpha f$ with $|\balpha|_\infty\leq m$
belong to $L_2 (\T^d)$ equipped with the norm
\begin{equation*}
\big\|\, f\, |H^m_{\mix}(\T^d)\big\|  := \Bigg(\sum_{\balpha \in \N_0^d,\, |\balpha|_{\infty} \leq m} \, \big\|\, D^{\balpha}f\, |L_2(\T^d)\big\|^2\Bigg)^{1/2}\, .
\end{equation*}  
Then $H^m_{\mix}(\T^d) = H^{m,r}_{\mix}(\T^d)$ for all $r$ in the sense of equivalent quasinorms.
If $m=1$, then we have 
$
\|\, \cdot \, |H^{1,2}_{\mix}(\T^d)\|= \|\, \cdot\, |H^1_{\mix}(\T^d)\|\, .
$
If $m \ge 2$, then the norm $\|\, \cdot\, |H^m_{\mix}(\T^d)\|$ itself does not belong to the family of norms
$\|\, \cdot\, |H^{m,r}_{\mix}(\T^d)\|$, $0 < r \le \infty$.
But the choice $r=2m$ leads to the following standard norm  
\begin{equation*}
\big\|\, f\, |H^{m,2m}_{\mix}(\T^d)\big\|= \Bigg(\sum_{\balpha \in \{0,m\}^d} \big\| \, D^{\balpha}f\, |L_2
(\T^d)\big\|^2\Bigg)^{1/2} \, ,
\end{equation*}
see \cite{KSU15}.
For a recent survey on the behavior of $s$-numbers of embeddings 
of these spaces into $L_p (\T^d)$ we refer to  \cite{DTU18B}. 

\begin{remark}
There is a different class of weights which would be of interest. For $0 < s < \infty$ and $0< r\le \infty$,
let 
\begin{align*} 
	\widetilde{\omega}_{s,r}(\bk) & :=  \bigg(1+ \sum_{i=1}^d |k_i|^r\bigg)^{s/r}\, , \qquad 0 < r < \infty\, , \\
	\widetilde{\omega}_{s,\infty}(\bk) & := \max \big(1,|k_1|, \ldots \, , |k_d|\big)^{s}
\end{align*}
$\bk \in \Z^d$.
The corresponding Sobolev spaces $H^{s,r}(\T^d):={F}_{\tilde{\omega}_{s,r}}(\T^d)$ are the standard periodic Sobolev spaces of fractional smoothness $s$,
the classes $\mathcal{A}^{s,r} (\T^d):=\mathcal{A}_{\tilde{\omega}_{s,r}}(\T^d)$ would be the natural counterparts.
The investigation of the behavior of the $s$-numbers for the embeddings $\mathcal{A}^{s,r} (\T^d) \hookrightarrow \mathcal{A}(\T^d)$
and $\mathcal{A}^{s,r} (\T^d) \hookrightarrow L_p(\T^d)$ will be postponed.
\end{remark}


\subsection{$s$-Numbers of $id: ~ \mathcal{A}^{s,r}_{\mix}(\T^d) \to \mathcal{A}(\T^d)$}


Here we are in a convenient situation. 
Based on \cite{KSU15,KSU20} and \eqref{ws-12b} all hard work is already done. 
First we recall a further result, see \cite{KSU15}.

\begin{prop}\label{alt}
	Let $0<s<\infty$ and $0<r\leq \infty$. Let $s_n \in \{x_n,b_n,d_n,a_n\}$, $n \in \N$.
	Then it holds
	\begin{equation*} 
		\lim\limits_{n\to \infty} \frac{s_n\big(id: H^{s,r}_{\mix}(\T^d)\to L_2(\T^d)\big)}{n^{-s}(\ln n)^{s(d-1)}}= 
		\bigg( \frac{2^d}{(d-1)!}\bigg)^s\,.
	\end{equation*}
\end{prop} 
In \cite{KSU15} this is formulated for $a_n$ only. 
But we are in a Hilbert space situation, see Lemma \ref{sing}. 
Let us mention that in the particular case $r=1$ this proposition follows also from a two sided estimate given in  \cite[Theorem 4.6]{ChD16},
see also the older arXiv-version \cite[Theorem 6.5]{ChD13}.

\begin{cor}\label{cor-A-A}
 Let $0 < s < \infty $, $0 <  r \le \infty$ and $s_n \in \{x_n, b_n, d_n, a_n\}$. 
Then
\begin{equation*}
 \lim_{n \to \infty} \, \frac{s_{n} \big(id: \, \mathcal{A}^{s,r}_{\mix} (\T^d) \to \mathcal{A} (\T^d)\big)}{n^{-s}(\ln n)^{(d-1)s} \, } =
\bigg(\frac{2^{d}}{(d-1)!}\bigg)^s \, .
\end{equation*}
\end{cor}

\begin{proof}
 The proof is a direct consequence of a corresponding result  for 
$a_{n} (id: \, {H}^{s,r}_{\mix} (\T^d) \to {L}_2 (\T^d))$ in Proposition \ref{alt}, and Corollary \ref{cor1}. 
\end{proof}

The asymptotic behavior of the $s$-numbers $s_{n} \big(id: \, \mathcal{A}^{s,r}_{\mix} (\T^d) \to \mathcal{A} (\T^d)\big)$
is quite different from the behavior for small $n$, say $n \le 2^d$.
This range is called the preasymptotic region. In case $d$ is large this is the only region where such a result could be 
of practical relevance (e.g., as a benchmark).

\begin{cor}
Let $d\ge 3$, $0<s<\infty$ and  $1\le r < \infty$.
Let $s_n \in \{x_n, b_n, d_n, a_n\}$.
For all $n \ge 2$ it holds
\beqq
s_n\big(id:~\mathcal{A}^{{s},{r}}_{\mix}(\T^d) \to \mathcal{A}(\T^d)\big) \leq
\bigg(\frac{C(d)}{n}\bigg)^{\frac{s}{r\, (1+\log_2(d-1))}}\,,
\eeqq
where $C(d)$ is defined as
\beqq
 C(d):= \bigg[ 1+ \frac{1}{d-1} \Big(1 + \frac{2}{\log_2 (d-1)}\Big)\bigg]^{d-1}\, .
\eeqq
\end{cor}

\begin{proof}
Again the proof follows from  Corollary \ref{cor1} and a corresponding result for 
$a_{n} (id: \, {H}^{s,r}_{\mix} (\T^d) \to {L}_2 (\T^d))$, see  \cite{KSU15,KSU20}. 
\end{proof}

Obviously  $\lim_{d\to \infty} \, C(d)= e$.
That makes clear that in the preasymptotic range there is an essential change in the approximation properties.
The asymptotically optimal rate 
$n^{-s} \, (\ln n)^{s(d-1)}$ is replaced by ${n}^{-\frac{s}{r\, (1+\log_2(d-1))}}$.
The larger $d$, the smaller the convergence rate is.
In addition there is a tremendous influence of the special norm chosen, i.e., of the parameter $r$.
If $r\to \infty$, then the rate is getting worse.
Here is the result in the limiting situation $r= \infty$.

\begin{cor}\label{unendl}
Let $0<s<\infty$ and $s_n \in \{x_n, b_n, d_n, a_n\}$.
Then
\[
 s_n \big(id:\, \mathcal{A}^{{s},{\infty}}_{\mix}(\T^d)\to \mathcal{A}(\T^d)\big) = 1\, , \qquad n=1, 2, \ldots , 3^d\, .
\]
\end{cor}

\begin{proof}
The proof follows from  Corollary \ref{cor1} and a corresponding result for 
$a_{n} (id: \, {H}^{s,\infty}_{\mix} (\T^d) \to {L}_2 (\T^d))$, see  \cite[Theorem 4.3]{KSU20}. 
\end{proof}


\subsection{$s$-Numbers of $id: ~ \mathcal{A}^{s,r}_{\mix}(\T^d) \to L_p(\T^d)$, $2\leq p\leq \infty$}


Theorem  \ref{thm-general} and Proposition \ref{alt} and  will be used to derive the following. 

 \begin{theorem}\label{thm:a-mix} Let $0<s<\infty$ and $0<r\leq \infty$. 
\\
{\rm (i)} In case of Bernstein or Weyl numbers we have 
\begin{equation*}
\lim\limits_{n\to \infty} \frac{ v_n\big(id: \mathcal{A}_{\mix}^{s,r}(\T^d) \to L_2(\T^d)\big)}{n^{-s-\frac{1}{2}}(\ln n)^{s(d-1)}}= 
\sqrt{ 2s+1}\bigg( \frac{2^d}{(d-1)!}\bigg)^s \,.
\end{equation*}
{\rm (ii)}
For approximation or Kolmogorov numbers it holds
\begin{equation*}\label{ws-18}
\lim\limits_{n\to \infty} \frac{ u_n\big(id: \mathcal{A}_{\mix}^{s,r}(\T^d) \to L_2(\T^d)\big)}{n^{-s}(\ln n)^{s(d-1)}}
= \bigg(\frac{2s}{2s+1}\bigg)^s\bigg( \frac{2^d}{(d-1)!}\bigg)^s.
\end{equation*}
 \end{theorem}

Now we turn to  the approximation and Kolmogorov numbers of  $id: ~\mathcal{A}^{s,r}_{\mix}(\T^d) \to L_p(\T^d)$ with $2<p\leq \infty$.

 \begin{theorem} \label{lunendlich}
Let $0<s<\infty$ and $0<r\leq \infty$. Then for any $\varepsilon >0$ there exists some  $n_0\in \N$ such that
\begin{equation} \label{ws-41}
\begin{aligned}
		\bigg(\frac{2s}{2s+1}\bigg)^s\bigg( \frac{2^d}{(d-1)!}\bigg)^s -\varepsilon  \leq  \frac{ u_n\big(id: 
			\mathcal{A}^{s,r}_{\mix}(\T^d) \to L_p(\T^d)\big)}{n^{-s}(\ln n)^{s(d-1)} }  
		\leq 
	 \bigg( \frac{2^d}{(d-1)!}\bigg)^s +\varepsilon\,
\end{aligned}
\end{equation}
holds for all $n\geq n_0$.
\end{theorem}

\begin{proof} For $2<p\leq \infty$, in view of the chain of embeddings
	$$\mathcal{A}^{s,r}_{\mix}(\T^d)\hookrightarrow \mathcal{A}(\T^d)\hookrightarrow L_p(\T^d)\hookrightarrow L_2(\T^d)$$
	we conclude
	\be\label{ws-51}
	u_n\big(id: \mathcal{A}_{\mix}^{s,r}(\T^d) \to L_2(\T^d)\big)\le 
	u_n\big(id: \mathcal{A}_{\mix}^{s,r}(\T^d) \to L_p(\T^d)\big)
	\le 
	u_n\big(id: \mathcal{A}_{\mix}^{s,r}(\T^d) \to \mathcal{A}(\T^d)\big)\, .
	\ee
Employing Theorem \ref{thm:a-mix} (ii) yields 
\[
\frac{u_n\big(id: ~\mathcal{A}^{s,r}_{\mix}(\T^d) \to L_2(\T^d))}{n^{-s}(\ln n)^{s(d-1)} }   \xrightarrow[n \to \infty]{} 
\bigg(\frac{2s}{2s+1}\bigg)^s\bigg( \frac{2^d}{(d-1)!}\bigg)^s\, .
\]
On the other hand we know
\[
 \frac{\sigma_n}{n^{-s}(\ln n)^{s(d-1)}} = \frac{u_n\big(id: \mathcal{A}^{s,r}_{\mix}(\T^d)\to \mathcal{A}(\T^d)\big)}{n^{-s}(\ln n)^{s(d-1)}}
 \xrightarrow[n \to \infty]{} \bigg( \frac{2^d}{(d-1)!}\bigg)^s\,,
\]
see Corollary \ref{cor-A-A}. Now combining these two facts with \eqref{ws-51} we obtain the claim.
\end{proof}

\begin{remark}
\rm
(i) One of the main results in \cite{CKS16} reads as follows.
Let $d\in \mathbb{N}$ , $s>1/2$ and $0<r\le \infty$. Then 
\[
\lim_{n \rightarrow \infty}\frac{  a_n(I_d : H^{s,r} _{\rm{mix}}(\mathbb{T}^d) \to 
L_\infty (\mathbb{T}^d))}{n^{-s+1/2}(\ln n)^{(d-1)s}} = 
\frac{1}{\sqrt{2s-1}}\bigg(\frac{2^d}{(d-1)!}\bigg)^s \, .
\]
Comparing this behavior with the result  obtained in Theorem \ref{lunendlich}
we notice an essential difference. Indeed, we have that for any $\varepsilon >0 $  there exists some $n_0 $
such that 

\[
\frac{1}{\sqrt{n}}   \le \frac{1}{\sqrt{n}}  \, 
 \frac{a_n(I_d : H^{s,r} _{\rm{mix}}(\mathbb{T}^d))\to L_\infty (\T^d))}{a_n(I_d : \mathcal{A}^{s,r} _{\rm{mix}}(\mathbb{T}^d) \to L_\infty (\T^d))}
\le  \bigg(1 + \frac{1}{2s}\bigg)^s \frac{1}{\sqrt{2s-1}} + \varepsilon
\]
holds for all $n\ge n_0$. For the  first inequality we used \eqref{ws-40} and \eqref{ws-04}.
\\
(ii) The left-hand side of inequality \eqref{ws-41} can be simplified. 
Elementary analysis yields that  for all $\varepsilon >0$ there exists some  $n_0\in \N$ such that
\begin{equation*} 
\begin{aligned}
		\frac{1}{\sqrt{e}}\, \bigg( \frac{2^d}{(d-1)!}\bigg)^s -\varepsilon  \leq  \frac{ u_n\big(id: 
			\mathcal{A}^{s,r}_{\mix}(\T^d) \to L_\infty(\T^d)\big)}{n^{-s}(\ln n)^{s(d-1)} }  
		\leq 
	 \bigg( \frac{2^d}{(d-1)!}\bigg)^s +\varepsilon\,
\end{aligned}
\end{equation*}
is true  for all $n\geq n_0$.
\\
{\rm (iii)} Above we have mentioned the identity
\[
a_n\big(id: F_\omega(\T^d) \to L_\infty(\T^d)\big) = 
a_n\big(id: F_\omega(\T^d) \to \mathcal{A}(\T^d)\big)\, , \qquad n \in \N\, .
\]
see \eqref{ws-61}.
Up to now it is not clear whether the counterpart 
\[
a_n\big(id: \mathcal{A}^{s,r}_{\mix}(\T^d) \to L_\infty(\T^d)\big) = 
a_n\big(id: \mathcal{A}^{s,r}_{\mix}(\T^d) \to \mathcal{A}(\T^d)\big)
\]
holds.
But at least we know that for any $\varepsilon >0$ there exists some  $n_0\in \N$ such that
\begin{equation*} 
\begin{aligned}
		1  \leq  \frac{ a_n\big(id: 
			\mathcal{A}^{s,r}_{\mix}(\T^d) \to \mathcal{A}(\T^d)\big)}{a_n\big(id: 
			\mathcal{A}^{s,r}_{\mix}(\T^d) \to L_\infty(\T^d)\big)}  
		\leq \bigg(\frac{2s+1}{2s}\bigg)^s  +\varepsilon< \sqrt{e}+\varepsilon\,
\end{aligned}
\end{equation*}
holds for all $n\geq n_0$. This follows from 
\[
a_n\big(id: \mathcal{A}^{s,r}_{\mix}(\T^d) \to L_\infty(\T^d)\big)\le 
 a_n\big(id: \mathcal{A}^{s,r}_{\mix}(\T^d) \to \mathcal{A}(\T^d)\big)
\]
and Theorem \ref{lunendlich}(i).
\end{remark}


\subsection{$s$-Numbers of $id: \mathring{S}^m_{\infty}W(\T^d)\to L_2(\T^d)$}


Denote $\mathring{\Z}^d := \{\bk \in \Z^d: k_j\not=0, j=1,\ldots,d\}$. Let $m \in \N$. In this section we deal with the weight
$
\omega_m =(\omega_m(\bk))_{\bk\in \mathring{\Z}^d}$, 
where $\omega_m(\bk) := \prod_{j=1}^d |k_j|^{m}$, $\bk\in \mathring{\Z}^d$. 
We consider  the functions $f\in L_1(\T^d)$ having mean zero in every direction $j=1,\ldots,d$, i.e.,
\begin{equation}\label{eq-k00}
\int_{\T} f(\bx)\dd x_j=0,\qquad j=1,\ldots,d.
\end{equation}
We define  $\mathring{\mathcal{A}}^{m}_{\mix}(\T^d)$ as the collection of all functions $f$ in $C(\T^d)$ satisfying \eqref{eq-k00}
and
\[
\|f|\mathring{\mathcal{A}}^{m}_{\mix}(\T^d)\|:= \sum_{\bk\in \mathring{\Z}^d}  \omega_m(\bk) |\hat{f}(\bk)|< \infty \, .
\]
Let $F_{m}(\bx)$ be the multidimensional analogue of the Bernoulli kernel
\begin{equation*} \label{eq-Bernoulli}
F_{m}(\bx) = 2^d\, \sum_{\bk \in \N^d} \prod_{j=1}^d k_j^{-m} \cos\Big(k_jx_j-\frac{m\pi}{2}\Big).
\end{equation*}
We define $\mathring{H}^{m}_{\mix}(\T^d)$ the set of all functions $f$
which can be represented  in the following form
\begin{equation}\label{eq-representation}
f(\bx) = F_{m}(\bx) * \varphi(\bx)=(2\pi)^{-d}\int_{\T^d} \varphi(\by)F_{m}(\bx-\by)\dd \by ,
\end{equation}
where $\varphi \in L_2(\T^d)$. The norm of $f\in \mathring{H}^{m}_{\mix}(\T^d)$ is given by
$$
\|f|\mathring{H}^{m}_{\mix}(\T^d)\| := \|\varphi |L_2(\T^d)\|.
$$
Similarly, we define   $\mathring{S}^m_{\infty}W(\T^d)$ as the set  of
all functions $f\in L_\infty(\T^d)$ which can be represented in the form \eqref{eq-representation} with $\varphi\in L_\infty(\T^d)$. The norm of $f\in \mathring{S}^m_{\infty}W(\T^d)$ is defined by
\begin{equation*}
	\|f|\mathring{S}^m_{\infty}W(\T^d)\|: =  \|\varphi|L_\infty(\T^d)\|.
\end{equation*}
Because of  $\varphi =D^{(m,\ldots,m)} f$ we conclude 
\begin{equation*}
	\| f|\mathring{H}^{m}_{\mix}(\T^d)\| =\|D^{(m,\ldots,m)} f|L_2(\T^d)\| = \Bigg(\sum_{\bk\in \mathring{\Z}^d} \omega_m(\bk)^2|\hat{f}(\bk)|^2\Bigg)^{1/2} 
\end{equation*}
and
\begin{equation*}
	\|f|\mathring{S}^m_{\infty}W(\T^d)\|= \|D^{(m,\ldots,m)} f|L_\infty(\T^d)\|\,.
\end{equation*}

\begin{lemma}\label{einb1}
 Let $m \in \N$. 
Then we have the chain of continuous embeddings 
\[
\mathring{\mathcal{A}}^{m}_{\mix}(\T^d)  \hookrightarrow \mathring{S}^m_{\infty}W(\T^d)\hookrightarrow \mathring{H}^{m}_{\mix}(\T^d)
\, .
\]
The norm of these embedding operators is always $1$.
\end{lemma}

\begin{proof} Let $f \in \mathring{\mathcal{A}}^{m}_{\mix}(\T^d) $. Then $D^{(m,\ldots,m)} f$ is a continuous function 
which coincides with its Fourier series given by 
$\sum_{\bk \in \mathring{\Z}^d} \hat{f} (\bk)\,  e^{{\rm i}\bk \bx} \big(\prod_{\ell=1}^d ({\rm i}k_\ell)^{m} \big)$.
Furthermore, we have the obvious estimate
\begin{align*}
\Bigg|\sum_{\bk \in \mathring{\Z}^d} \hat{f} (\bk)\,  e^{{\rm i}\bk \bx} \bigg(\prod_{\ell=1}^d ({\rm i}k_\ell)^{m} \bigg) \Bigg|
& \le   \sum_{\bk \in \mathring{\Z}^d} |\hat{f} (\bk)|\, \prod_{\ell=1}^d |k_\ell|^{m} 
=  \| \, f\, |\mathring{\mathcal{A}}^{m}_{\mix}(\T^d) \|\, . 
\end{align*}
This proves the first embedding. The second embedding is obvious.
\end{proof}

We have an analogue of Proposition \ref{alt}.

\begin{prop}\label{prop-o}
	Let $m \in \N$ and $s_n \in \{x_n,b_n,d_n,a_n\}$, $n \in \N$.
	Then it holds
\begin{equation*} 
\lim\limits_{n\to \infty} \frac{s_n\big(id: \mathring{H}^{m}_{\mix}(\T^d)\to L_2(\T^d)\big)}{n^{-m}(\ln n)^{m(d-1)}}= 
\bigg( \frac{2^d}{(d-1)!}\bigg)^m\,.
	\end{equation*}
\end{prop}

\begin{proof} Recall that the sequence $s_n\big(id: \mathring{H}^{m}_{\mix}(\T^d)\to L_2(\T^d)\big)$ is the nonincreasing rearrangement of 
$(1/\omega_m(\bk))_{\bk\in \mathring{\Z}^d}$, see Lemma \ref{sing}. 
We put
\[
 c(r,d):= \bigg|\bigg\{\bk \in \mathring{\Z}^d: ~\prod_{j=1}^d |k_j| 
 \leq r \bigg\}\bigg|, \quad r \in \N \, .
\]
Using the estimates given in  \cite[Theorems 3.4 and 3.5]{ChD16} (with $a=1$) we find
\be\label{ch}
\frac{2^d}{(d-1)!}\,  \frac{r(\ln r)^d}{\ln r+d} < c(r,d) < \frac{2^d}{(d-1)!} \, \frac{r(\ln r +d\ln 2)^d}{\ln r+d\ln 2 +d-1}
\ee
for $r>r*>1$.
Now we may proceed as in the proof of \cite[Theorem 4.3]{KSU15}. By definition of $c(r,d)$ we know
\[
 s_n := s_n \big(id: \mathring{H}^{m}_{\mix}(\T^d)\to L_2(\T^d)\big) = r^{-m}
 \qquad \mbox{if}\quad c(r-1,d) <n \le c(r,d)\, , \qquad n \ge 2\, .
\]
Clearly, $\lim_{r\to \infty} c(r, d) = \infty$;
moreover, the sequence $n\, (\ln n)^{-(d-1)}$ is increasing for 
$n > e^{d-1}$. Hence we obtain
for sufficiently large $r \in  \N$ the two-sided inequality
\[
\frac{ c(r - 1, d)}{r\, (\ln c(r-1,d))^{d-1}}
\le s_n^{1/m} \, \frac{n}{(\ln n)^{d-1}}\le 
\frac{c(r, d)}{r\, (\ln c(r, d))^{d-1}}\, .
\]
Applying \eqref{ch} the claim follows. 
\end{proof}

Our main result in this section reads as follows.

\begin{theorem}  \label{satz2}
Let $m \in \N$. 
\\
{\rm (i)} Let $(\sigma_k)_{k\in \N}$ be of the  non-increasing rearrangement of the sequence $(1/\omega_m(\bk))_{\bk \in \mathring{Z}^d}$. Then the following identity takes place
\begin{equation*} \label{eq:bw-cmix}
v_n\big(id: \mathring{S}^m_{\infty}W(\T^d) \to L_2(\T^d)\big) 
= \Bigg(\sum_{k=1}^n\sigma_k^{-2}\Bigg)^{-1/2}\,.
\end{equation*}
{\rm (ii)} In case of Weyl and Bernstein numbers the asymptotic constant is given by
\begin{equation*}
\lim\limits_{n\to \infty} \, \frac{ v_n\big(id: \mathring{S}^m_{\infty}W(\T^d) \to L_2(\T^d)\big)}{n^{-m-\frac{1}{2}}(\ln n)^{m(d-1)}}=
\sqrt{ 2m+1}\, \bigg( \frac{2^d}{(d-1)!}\bigg)^m \, .
\end{equation*}
{\rm (iii)} For all $\varepsilon >0$  there exists some $n_0\in \N$ such that
\begin{equation} \label{eq:ad-cmix}
\begin{aligned} 
\bigg(\frac{2m}{2m+1}\bigg)^m\bigg( \frac{2^d}{(d-1)!}\bigg)^m -\varepsilon  
\leq  \frac{ u_n\big(id: \mathring{S}^m_{\infty}W(\T^d) \to L_2(\T^d)\big)}{n^{-m}(\ln n)^{m(d-1)} }  
\leq    \bigg( \frac{2^d}{(d-1)!}\bigg)^m + \varepsilon\,
\end{aligned} 
\end{equation} 
holds for all $n\geq n_0$.
\end{theorem}

\begin{proof}
{\it Step 1.} Proof of (i). From the embedding
$\mathring{\mathcal{A}}^{m}_{\mix}(\T^d) \hookrightarrow  \mathring{S}^m_{\infty}W(\T^d)$,  with operator norm $1$, see 
Lemma \ref{einb1}, we conclude by means of \eqref{ws-04} that
\begin{equation*}
\begin{aligned} 
v_n\big(id: \mathring{\mathcal{A}}^{m}_{\mix}(\T^d) \to L_2(\T^d)\big)
	&
	 \leq
	  \|id: \mathring{\mathcal{A}}^{m}_{\mix}(\T^d) \to \mathring{S}^m_{\infty}W(\T^d) \|\cdot v_n\big(id: \mathring{S}^m_{\infty}W(\T^d) \to L_2(\T^d)\big)
	  \\
	& 
	= 
	v_n\big(id: \mathring{S}^m_{\infty}W(\T^d) \to L_2(\T^d)\big)\,.
\end{aligned}
\end{equation*}
Theorem \ref{thm:bw} yields  the lower estimate. 
\\
Since the target space is a Hilbert space we know that  
$$x_n\big(id: \mathring{S}^m_{\infty}W(\T^d) \to L_2(\T^d)\big)
\leq 
b_n\big(id: \mathring{S}^m_{\infty}W(\T^d) \to L_2(\T^d)\big),$$ see \eqref{weylb}. Consider the  operators 
$A: \mathring{S}^m_{\infty}W(\T^d) \to L_\infty(\T^d)$ and $K: L_\infty(\T^d)\to L_2(\T^d)  $
defined by $A(f):=\varphi$ and $K(\psi):=F_m*\psi$, respectively.
Here  $\varphi$ is as in \eqref{eq-representation}. 
Then $id = K \circ A$ and we obtain 
\begin{equation}\label{eq-k04}
\begin{aligned}
b_n\big(id: \mathring{S}^m_{\infty}W(\T^d) \to L_2(\T^d)\big)
& \leq \|A:\mathring{S}^m_{\infty}W(\T^d) \to L_\infty(\T^d)\| \cdot b_n(K:L_\infty(\T^d) \to L_2(\T^d)  ) 
\\
&=  b_n(K:L_\infty(\T^d) \to L_2(\T^d)  ) .
\end{aligned}
\end{equation}
Let $m$ be an even number.
Then $F_m$ can be written as
\begin{equation*}
	F_{m}(\bx) = 2^d\sum_{\bk \in \N^d} \prod_{j=1}^d k_j^{-m} (-1)^{\frac{m}{2}} \frac{e^{{\rm i} k_jx_j}+e^{- {\rm i} k_jx_j}}{2} = (-1)^{\frac{dm}{2}} \sum_{\bk \in \mathring{\Z}^d} \omega_m(\bk)^{-1}e^{{\rm i} \bk \bx}.
\end{equation*}
This implies that  $K$ is an operator as given  in \cite[Example 4]{Par95}. Therefore we may apply Theorem 9 in 
\cite{Par95} and  obtain 
$$
b_n\big(K:L_\infty(\T^d) \to L_2(\T^d)\big) = \Bigg(\sum_{k=1}^n\sigma_k^{-2}\Bigg)^{-1/2}.
$$
This identity combined with  \eqref{eq-k04} proves the upper bound when $m$ is even. The case $m$  being an odd number can be  carried
out similarly. 
\\
{\em Step 2.} Proof of (ii).
Part (ii) becomes a direct consequence of (i) in combination with  Theorem \ref{thm-general}(i) and Proposition \ref{prop-o}.
\\
{\it Step 3.} Proof of (iii).
For the  lower bound in \eqref{eq:ad-cmix} we may use the same argument as in Step 1. Then we have
$$
u_n\big(id: \mathring{\mathcal{A}}^{m}_{\mix}(\T^d) \to L_2(\T^d)\big)
\leq 
u_n\big(id: \mathring{S}^m_{\infty}W(\T^d) \to L_2(\T^d)\big)\,.
$$
Concerning the upper bound we shall employ the embeddings $\mathring{S}^m_{\infty}W(\T^d) \hookrightarrow  \mathring{H}^{m}_{\mix}(\T^d)\hookrightarrow L_2(\T^d)$.
We obtain  
\begin{equation*}
\begin{aligned} 
u_n\big(id: \mathring{S}^m_{\infty}W(\T^d) \to L_2(\T^d)\big) 
&
 \leq 
 \|id: \mathring{S}^m_{\infty}W(\T^d) \to \mathring{H}^{m}_{\mix}(\T^d)\| 
 \cdot 
 u_n\big(id: \mathring{H}^{m}_{\mix}(\T^d) \to L_2(\T^d)\big)\\
&
\leq  u_n\big(id: \mathring{H}^{m}_{\mix}(\T^d) \to L_2(\T^d)\big)\,.
\end{aligned}
\end{equation*}
In view of Theorem \ref{thm-general} and Proposition \ref{prop-o} we have finished the proof.
 \end{proof}


\section{$s$-Numbers of $id: \mathcal{A}^{s,2}_{\mix}(\T^d)\to H^1(\T^d)$}
\label{ex3}


In this section, we are interested in the asymptotic constants for $s$-numbers of the embeddings $ \mathcal{A}^{s,2}_{\mix}(\T^d)\hookrightarrow  H^1(\T^d)\big)$, $s> 1$. 
Here $H^1(\T^d)$ is the  isotropic periodic Sobolev space with smoothness $1$ equipped with the norm 
\begin{align*}
 \|\, f\, |H^1(\T^d)\|:&= \Bigg(\sum_{\bk \in \Z^d} \bigg(1+\sum_{j=1}^d |k_j|^2\bigg) |\hat{f}(\bk)|^2\Bigg)^{1/2} 
 \\
 &= 
\Bigg(\|\, f\, |L_2(\T^d)\|^2 + \sum_{j=1}^d \Big\|\, \frac{\partial f}{\partial x_j}\, \Big|L_2(\T^d)\Big\|^2\Bigg)^{1/2}
\, .
\end{align*}
Clearly, if $d\ge 2$ we have $H^1_{\mix}(\T^d)\neq H^1(\T^d)$, in fact the 
embedding $H^1_{\mix}(\T^d) \hookrightarrow  H^1(\T^d)$ is proper. As a preparation  we  shall investigate 
the asymptotic constant of $s_n\big(id: H^{s,2}_{\mix}(\T^d)\to H^1(\T^d)\big)$, $s> 1$. 
Because we are in a Hilbert space situation it will be enough to deal  with the approximation numbers.
We define a weight $\omega$ by
\begin{equation}\label{eq-omega}
\omega(\bk):= \frac{\prod_{j=1}^d(1+|k_j|^2)^{s/2}}{\big(1+\sum_{j=1}^d |k_j|^2\big)^{1/2}},\qquad \bk\in \Z^d\, .
\end{equation}
In addition we need the diagonal operator $D_\omega \xi := (\xi(\bk)/\omega (\bk))_{\bk \in \Z^d}$, $\xi = (\xi(\bk))_{\bk \in \Z^d}$,
see \eqref{ws-12}. Now we observe that 
\begin{equation}\label{eq-f-omega}
a_n\big(id: H^{s,2}_{\mix}(\T^d)\to H^1(\T^d)\big) =
 a_n\big(D_\omega : \ell_2(\Z^d)\to \ell_2(\Z^d)\big)=a_n\big(id: F_\omega(\T^d)\to L_2(\T^d)\big) ,
\end{equation}
 see Lemma \ref{sing}.
Rearranging non-increasingly the sequence $(1/\omega(\bk))_{\bk\in \Z^d}$ with the outcome denoted by $(\sigma_n)_{n\in \N}$, 
we obtain $\sigma_n = a_n\big(id: H^{s,2}_{\mix}(\T^d)\to H^1(\T^d)\big)$. 

The asymptotic order of $a_n\big(id: H^{s,2}_{\mix}(\T^d)\to H^1(\T^d)\big)$ is well-known, see \cite{DU13,BDSU16}. It holds
\begin{align*}
	C_1(s,d)n^{-s+1} \leq a_n\big(id: H^{s,2}_{\mix}(\T^d)\to H^1(\T^d)\big) \leq C_2(s,d)n^{-s+1}, \qquad n \in \N\, , 
\end{align*}
with constants $C_1(s,d)$ and $C_2(s,d)$ depending on $s,d$. Notice that there is no logarithmic factor anymore.
Several preasymptotic estimates for $a_n\big(id: H^{s,2}_{\mix}(\T^d)\to H^1(\T^d)\big)$ may be found in \cite{KSU20}. 
Our result for the asymptotic constant of $a_n\big(id: H^{s,2}_{\mix}(\T^d)\to H^1(\T^d)\big)$ reads as follows.

\begin{theorem}\label{thm-dk} Let $s >1$ and define 
\begin{equation}\label{eq-S}
S:=\sum_{k=1}^{\infty}\dfrac{1}{(k^2+1)^{\frac{s}{2(s-1)}}}.
\end{equation}
Then we have
$$\lim _{n \to +\infty} \frac{a_n\big(id: H^{s,2}_{\mix}(\T^d)\to H^1(\T^d)\big)}{n^{-s+1}}=(2d)^{s-1}(2S+1)^{(s-1)(d-1)}.$$ 
\end{theorem}
Theorem \ref{thm-dk} in case $d=1$  is almost obvious and represents an old result of Kolmogorov, see \cite{Kol36}.
However, it can be found also in \cite[Theorem.~4.14]{KSU14}. In the following we will assume that $d\geq 2$. 
For $s>1$, $r\in \N$ and $\omega(\bk)$ as in \eqref{eq-omega} we define
$$
\mathcal M(r,d):=\big\{\bk \in \mathbb Z^d: \omega(\bk)^2 \le (1+r^2)^{s-1}\big\}
$$
and $C(r,d):=|\mathcal M(r,d)|$. The following lemma will be crucial for the proof of Theorem \ref{thm-dk}. 

\begin{lemma} \label{lem-refree} Let $S$ be given in \eqref{eq-S}. Then it holds
	$$\lim \limits_{r \to +\infty} \dfrac{C(r,d)}{r}=2d(1+2S)^{d-1}.$$
\end{lemma}

\begin{proof}{\it Step 1.} First observe the elementary implication
	\begin{align} \label{dad1}
		\bk \in \mathcal M(r,d)\ \ \Rightarrow\ \ \prod_{j=1}^d(1+k_j^2) \le 1+r^2. \end{align}
	Indeed, the trivial inequality $1+\sum_{j=1}^d k_j^2 \le \prod_{j=1}^d (1+k_j^2)$ implies 	$$(1+r^2)^{s-1} \ge \omega(\bk)^2 =\dfrac{\prod_{j=1}^d (1+k_j^2)^s}{1+\sum_{j=1}^d k_j^2} \ge \prod \limits_{j=1}^d (1+k_j^2)^{s-1}.$$
We consider the sets 
	\begin{align*}
		&\mathcal M_0(r,d):=\big\{\bk \in \mathcal M(r,d): \max \limits_{j=1,\ldots,d} |k_j|<r^{3/4}\big\}\  \\
		&\mathcal M_i(r,d):=\big\{\bk \in \mathcal M(r,d): |k_i|\ge r^{3/4}\big\},\ \  i=1 ,\ldots, d.
	\end{align*}
	Clearly $$\mathcal M(r,d)=\bigcup_{i=0}^d \mathcal M_i(r,d).$$
Obviously $\mathcal M_0(r,d)$ is disjoint from all other $\mathcal M_i(r,d)$, $i=1,\ldots,d$. Assuming that there is a $\bk \in \mathcal M_i(r,d) \cap \mathcal M_j(r,d)$ for $1 \le i \neq j \le d,$ the observation \eqref{dad1} would imply the contradiction $$1+r^2 \ge (1+k_i^2)(1+k_j^2) \ge (1+r^{3/2})^2 >1+r^2.$$
	Therefore $\mathcal M_i(r,d) \cap \mathcal M_j(r,d)=\emptyset$ for $0 \le i \neq j \le d$. 
	With the notation  $C_i(r,d):=|\mathcal M_i(r,d)|$, $i=0,\ldots,d$, we obtain \begin{align}
		\label{dad2}
		C(r,d)=C_0(r,d)+\sum \limits_{i=1}^d C_i (r,d)=C_0(r,d)+dC_d(r,d).
	\end{align}
{\it Step 2.}	Estimate of $C_0(r,d)$. 
	For $\tau>1$ we consider the following family of the hyperbolic crosses
	$$\mathcal{H}_d(\tau):=\bigg \{\bk \in \mathbb Z^d: \prod_{j=1}^d(1+k_j^2)^{1/2} \le \tau \bigg \}.$$
Recall the well-known estimate for its cardinality  $|\mathcal{H}_d(\tau)|=\mathcal{O}(\tau(\ln \tau)^{d-1})$ if $\tau$ is tending to infinity, see, 
e.g., formula  (0.1) on page 130 in \cite{T93b}.
For $\bk \in \mathcal M_0(r,d)$, we have
	\begin{align*}
		\prod_{j=1}^d (1+k_j^2)^s&=\omega(\bk)^2 \bigg(1+\sum \limits_{j=1}^d k_j^2 \bigg)\le (1+r^2)^{s-1}(1+dr^{3/2})=\mathcal O(r^{2s-1/2}).	
	\end{align*}
Let $\alpha = 1- \frac1{4s}<1.$	
Hence, there is a constant $c=c(r,d,s)>0$ such that 
	$$\mathcal M_0(r,d) \subseteq \mathcal H_d(c r^\alpha)=\bigg \{\bk \in \mathbb Z^d: \prod_{j=1}^d (1+k_j^2)^{1/2} \le cr^\alpha \bigg \}.$$
Since the cardinality of the hyperbolic cross $\mathcal H_d(cr^\alpha)$ is of order $\mathcal O(r^\alpha (\ln r)^{d-1})$, we get
\begin{equation}\label{eq-c0}
	\lim \limits_{r \to +\infty} \frac{C_0(r,d)}{r}=0.
\end{equation}
{\it Step 3.} Estimate of $C_d(r,d).$ \\
{\it Substep 3.1.} Estimate from above. For $\bl=(l_1,\ldots,l_{d-1}) \in  \mathbb{ Z}^{d-1}$  denote $$p(\bl):=\prod_{j=1}^{d-1}(1+l_j^2)^{-\frac{s}{2(s-1)}}.$$
	Then we have $$\sum \limits_{\bl \in \mathbb Z^{d-1}}p(\bl)=\bigg(\sum \limits_{m \in \mathbb Z}(1+m^2)^{-\frac{s}{2(s-1)}} \bigg)^{d-1}=(1+2S)^{d-1}.$$
	Using again \eqref{dad1}, we have for any $\bk \in \mathcal M_d(r,d)$
	$$1+\sum \limits_{j=1}^{d-1} k_j^2 \le \prod_{j=1}^{d-1}(1+k_j^2) \le \dfrac{1+r^2}{1+k_d^2} \le \dfrac{1+r^2}{1+r^{3/2}} \le r^{1/2}.$$
	On one hand this implies $\bk^*:=(k_1,\ldots,k_{d-1}) \in \mathcal H_{d-1}(r^{1/4})$ and on the other hand it shows that the coordinate $k_d$ satisfies
	\begin{align*}
		k_d^{2(s-1)} \le (1+k_d^2)^{s-1} &\le \dfrac{(1+r^2)^{s-1}}{\prod_{j=1}^{d-1}(1+k_j^2)^s}\cdot\dfrac{1+k_d^2+\sum_{j=1}^{d-1}k_j^2}{1+k_d^2}\\
		& \le (1+r)^{2(s-1)}\Big(1+\dfrac{r^{1/2}}{r^{3/2}}\Big) \prod_{j=1}^{d-1}(1+k_j^2)^{-s},
	\end{align*}
	equivalently $$|k_d| \le (1+r)\Big(1+\frac{1}{r}\Big)^{\frac{1}{2(s-1)}} \prod_{j=1}^{d-1}(1+k_j^2)^{-\frac{s}{2(s-1)}}=f(r)p(\bk^*),$$	
	where $f(r):=(1+r)(1+\frac{1}{r})^{\frac{1}{2(s-1)}}$. 
	Hence we have shown the inclusion
	$$\mathcal M_d(r,d) \subseteq \mathcal B(r,d):=\big \{\bk \in \mathbb Z^d: \bk^* \in \mathcal H_{d-1}(r^{1/4}),\ |k_d| \le f(r)p(\bk^*) \big \}.
	$$
Since $|\{ m \in \mathbb Z: |m| \le R\}| \le 2R+1,$ we obtain
\begin{equation}\label{eq-cd01}
	\begin{aligned}
		\dfrac{C_d(r,d)}{r} & \le \dfrac{|\mathcal B(r,d)|}{r} \le \sum \limits_{\bk^* \in \mathcal H_{d-1}(r^{1/4})} \dfrac{2f(r)p(\bk^*)+1}{r}\\
		&\le \dfrac{2f(r)}{r} \sum \limits_{\bk^* \in \mathbb Z^{d-1}}p(\bk^*)+\dfrac{|\mathcal H_{d-1}(r^{1/4})|}{r}  \underset{r \to \infty} \longrightarrow 2(1+2S)^{d-1},
	\end{aligned}
\end{equation}
	where we have used $\lim \limits_{r \to \infty} \frac{f(r)}{r}=1$ and $|\mathcal H_{d-1}(r^{1/4})|=\mathcal O(r^{1/4}(\ln r)^{d-2})$.
\\
\noindent
{\it Substep 3.2.} Estimate from below. We define
	$$\mathcal A(r,d):=\big\{\bk \in \mathbb Z^d: \bk^* \in \mathcal H_{d-1}(r^{1/8}),\ r^{3/4} \le |k_d| \le rp(\bk^*)-1\big\}  .$$
For $\bk \in \mathcal A(r,d)$, we have
	\begin{align*}
		\omega(\bk)^2&= \frac{\prod_{j=1}^d(1+k_j^2)^{s}}{1+\sum_{j=1}^d k_j^2} \le  \frac{\prod_{j=1}^d(1+k_j^2)^{s}}{1+ k_d^2}\\
		&=(1+k_d^2)^{s-1}\prod_{j=1}^{d-1}(1+k_j^2)^s \le (1+|k_d|)^{2(s-1)}\prod_{j=1}^{d-1}(1+k_j^2)^s\\
		&\le r^{2(s-1)}p(\bk^*)^{2(s-1)}\prod_{j=1}^{d-1}(1+k_j^2)^s=r^{2(s-1)}<(1+r^2)^{s-1},
	\end{align*}
	which shows that $\bk \in \mathcal M_d(r,d)$ and hence $\mathcal A(r,d) \subseteq \mathcal M_d(r,d)$. Now using the estimate 
	$$|\{m \in \mathbb Z: R_1 \le |m| \le R_2\}| \ge 2(R_2-R_1-1)$$ for $0<R_1<R_2$ and $|\mathcal H_{d-1}(r^{1/8})|=\mathcal O(r^{1/8} (\ln r)^{d-2})$ this implies
\begin{equation}\label{eq-cd02}
	\begin{aligned}
		\dfrac{C_d(r,d)}{r} &\ge \dfrac{|\mathcal A(r,d)|}{r} \ge \sum \limits_{\bk^* \in \mathcal H_{d-1}(r^{1/8})} \dfrac{2(rp(\bk^*)-r^{3/4}-2)}{r} 	\\
		&=2 \sum \limits_{\bk^* \in \mathcal H_{d-1}(r^{1/8})} p(\bk^*)-\dfrac{2(r^{3/4}+2)}{r}|\mathcal H_{d-1}(r^{1/8})|\\
		&\underset{r \to \infty} \longrightarrow 2 \sum_{\bk^* \in \mathbb Z^{d-1}} p(\bk^*)=2(1+2S)^{d-1}.
	\end{aligned}
\end{equation}
From \eqref{eq-cd01} and \eqref{eq-cd02} we get
	$$\lim \limits_{r \to +\infty} \frac{C_d(r,d)}{r}=2(1+2S)^{d-1}.$$
This, together with  \eqref{eq-c0} and  \eqref{dad2}, implies the desired result.
\end{proof}

\begin{proof}[Proof of Theorem \ref{thm-dk}]	
Recall that $a_n\big(id: H^{s,2}_{\mix}(\T^d)\to H^1(\T^d)\big)$ is the  nonincreasing rearrangement of the sequence 
$(1/\omega(\bk))_{\bk\in \Z^d}$, see Lemma \ref{sing}.
Observe that for any $r\in \N$ we have 
\[
 (1+r^2)^{\frac{s-1}{2}} \in \{\omega (\bk):~ \bk \in \Z^d\}\, .
\]
Furthermore,
for any  $n \in \mathbb N$, $n \ge 2$, there exists some $ r \in \mathbb N$ such that 
	\be\label{eq} 
	C(r-1,d) < n \le C(r,d)
	\ee
and
	\begin{align*}
	\dfrac{1}{(r^2+1)^{\frac{s-1}{2}}} & =  
	a_{C(r,d)}\big(id: H^{s,2}_{\mix}(\T^d)\to H^1(\T^d)\big)
	\le 
	a_n\big(id: H^{s,2}_{\mix}(\T^d)\to H^1(\T^d)\big) 
	\\
	& \le  a_{C(r-1,d)}\big(id: H^{s,2}_{\mix}(\T^d)\to H^1(\T^d)\big)=
	\dfrac{1}{((r-1)^2+1)^{\frac{s-1}{2}}}\, .
	\end{align*}
	This leads to
	\begin{align*}
		\frac{C(r-1,d)^{s-1}}{(r^2+1)^{\frac{s-1}{2}}} \le \frac{a_n\big(id: H^{s,2}_{\mix}(\T^d)\to H^1(\T^d)\big)}{n^{-s+1}} \le \frac{C(r,d)^{s-1}}{((r-1)^2+1)^{\frac{s-1}{2}}} \label{dl}. 
	\end{align*}
	Now  the desired result follows from Lemma \ref{lem-refree}.
\end{proof}

Our main result in this section reads as follows.

\begin{theorem}\label{thm:a-mix-h1} Let $d\in \N$, $s>1$ and $S$ be given in \eqref{eq-S}.
\\
{\rm (i)} In case of Bernstein or Weyl numbers we have 
	\begin{equation*}
		\lim\limits_{n\to \infty} \frac{ v_n\big(id: \mathcal{A}_{\mix}^{s,2}(\T^d) \to H^1(\T^d)\big)}{n^{-s+\frac{1}{2}}}= 
		\sqrt{ 2s+1} (2d)^{s-1}(2S+1)^{(s-1)(d-1)} \,.
	\end{equation*}
{\rm (ii)} For approximation or Kolmogorov numbers it holds
	\begin{equation*} 
		\lim\limits_{n\to \infty} \frac{ u_n\big(id: \mathcal{A}_{\mix}^{s,2}(\T^d) \to H^1(\T^d)\big)}{n^{-s+1}}
		= \bigg(\frac{2s}{2s+1}\bigg)^s (2d)^{s-1}(2S+1)^{(s-1)(d-1)}.
	\end{equation*}
\end{theorem}
\begin{proof}Let $\omega$ be given in \eqref{eq-omega}. From \eqref{eq-f-omega} and Theorem \ref{thm-dk} we obtain
\begin{equation}\label{eq-f-omega2}
\lim _{n \to +\infty} \frac{a_n\big(id: F_\omega(\T^d)\to L_2(\T^d)\big)}{n^{-s+1}}=(2d)^{s-1}(2S+1)^{(s-1)(d-1)}.
\end{equation}
Next we will show that
\begin{equation}\label{eq-Amix-Aomega}
s_n\big(id: \mathcal{A}_{\mix}^{s,2}(\T^d) \to H^1(\T^d)\big) = s_n\big(id: \mathcal{A}_\omega(\T^d)\to L_2(\T^d)\big),
\end{equation}
by using standard lifting arguments. We consider the diagram
	\begin{equation*}
	\begin{CD}
		\mathcal{A}_{\mix}^{s,2}(\T^d)  @ > id >> H^1(\T^d) \\
		@VV A V @AA B A\\
		\mathcal{A}_\omega(\T^d) @ > id >> L_2(\T^d) \, 
	\end{CD}
\end{equation*}
where the linear operators $A$ and $B$ 	are defined by
\begin{align*} 
\widehat{Af}(\bk) & : =   \bigg(1+\sum_{j=1}^d |k_j|^2\bigg)^{1/2}\hat{f}(\bk),  \qquad \bk \in \Z^d, \quad f\in 	\mathcal{A}_{\mix}^{s,2}(\T^d)\, ,  
\\ 
\widehat{Bg}(\bk) & :=   \bigg(1+\sum_{j=1}^d |k_j|^2\bigg)^{-1/2}\hat{g}(\bk),\qquad \bk \in \Z^d, \quad g\in L_2(\T^d)\, .
\end{align*}
It is obvious that $\|A\|=\|B\|=1$. Now by the ideal property of $s$-numbers, see \eqref{ws-04}, we obtain
$$
s_n\big(id: \mathcal{A}_{\mix}^{s,2}(\T^d) \to H^1(\T^d)\big) \leq  s_n\big(id: \mathcal{A}_\omega(\T^d)\to L_2(\T^d)\big).
$$
The reverse inequality  follows from the modified diagram
	\begin{equation*}
	\begin{CD}
		\mathcal{A}_{\mix}^{s,2}(\T^d)  @ > id >> H^1(\T^d)\\
		@AA {A}^{-1} A @VV {B}^{-1} V\\
	\mathcal{A}_\omega(\T^d) @ > id >> L_2(\T^d) \, .
	\end{CD}
\end{equation*}
Now the claims follow from Theorem \ref{thm-general}, \eqref{eq-f-omega2}, and \eqref{eq-Amix-Aomega}. 
\end{proof}


\noindent
{\bf Acknowledgments.} The  authors  wishes  to  thank    the  referees  for  their  careful  reading  of  the  paper,  their suggestions and clever remarks to improve the paper, in particular the proof of Theorems \ref{thm-general} and \ref{thm-dk}. The research of Van Dung Nguyen is funded by University of Transport and Communications (UTC) under grant number T2022-CB-008. The research of Van Kien Nguyen is funded by Vietnam 
National Foundation for Science and Technology Development (NAFOSTED) 
under grant number 102.01-2020.03. A part of this paper was done when Van Kien Nguyen was working at the Vietnam Institute for Advanced Study in Mathematics (VIASM). He would like to thank the VIASM for providing a fruitful research environment and working condition.


\end{document}